\documentclass{article}
\usepackage{amsmath}
\usepackage{amssymb}
\usepackage{amsfonts}
\usepackage{eucal}
\usepackage{url}
\usepackage{bm}
\usepackage[dvipdfmx]{hyperref}
\usepackage{mathrsfs}
\usepackage[dvipdfmx]{graphicx}
\usepackage{amscd}
\usepackage{amsthm}
\usepackage[all]{xy}
\usepackage{mathtools}
\DeclareMathOperator{\comod}{-comod}
\DeclareMathOperator{\End}{End}
\DeclareMathOperator{\Frac}{Frac}
\DeclareMathOperator{\Hom}{Hom}
\DeclareMathOperator{\id}{id}
\DeclareMathOperator{\Ind}{Ind}
\DeclareMathOperator{\ind}{ind}
\DeclareMathOperator{\Ker}{Ker}
\DeclareMathOperator{\cmod}{-mod}
\DeclareMathOperator{\ord}{ord}
\DeclareMathOperator{\pro}{pro}
\DeclareMathOperator{\Spec}{Spec}
\DeclareMathOperator{\PU}{PU}

\DeclareMathOperator{\rat}{rat}
\DeclareMathOperator{\un}{un}
\newcommand{\cA}{\mathcal{A}}

\newcommand{\cF}{\mathcal{F}}

\newcommand{\cJ}{\mathcal{J}}
\newcommand{\cO}{\mathcal{O}}

\newcommand{\fa}{\mathfrak{a}}
\newcommand{\fb}{\mathfrak{b}}
\newcommand{\fg}{\mathfrak{g}}

\newcommand{\fh}{\mathfrak{h}}
\newcommand{\fk}{\mathfrak{k}}
\newcommand{\fl}{\mathfrak{l}}
\newcommand{\fm}{\mathfrak{m}}

\newcommand{\fq}{\mathfrak{q}}
\newcommand{\fsl}{\mathfrak{sl}}
\newcommand{\ft}{\mathfrak{t}}
\newcommand{\fu}{\mathfrak{u}}

\newcommand{\bC}{\mathbb{C}}
\newcommand{\bF}{\mathbb{F}}

\newcommand{\bQ}{\mathbb{Q}}
\newcommand{\bR}{\mathbb{R}}
\newcommand{\bZ}{\mathbb{Z}}
\def\bfb{\boldsymbol{\mathfrak{b}}}
\def\bfg{\boldsymbol{\mathfrak{g}}}

\def\bfq{\boldsymbol{\mathfrak{q}}}
\def\bfsl{\boldsymbol{\mathfrak{sl}}}

\theoremstyle{plain}
\newtheorem{thm}{Theorem}[section]
\newtheorem{cor}[thm]{Corollary}
\newtheorem{lem}[thm]{Lemma}
\newtheorem{prop}[thm]{Proposition}
\theoremstyle{definition}

\newtheorem{defn}[thm]{Definition}
\newtheorem{ex}[thm]{Example}
\newtheorem{note}[thm]{Notation}

\newtheorem{rem}[thm]{Remark}
\newtheorem{var}[thm]{Variant}
\makeatletter
\newcommand*{\relrelbarsep}{.386ex}
\newcommand*{\relrelbar}{%
	\mathrel{%
		\mathpalette\@relrelbar\relrelbarsep
	}%
}
\newcommand*{\@relrelbar}[2]{%
	\raise#2\hbox to 0pt{$\m@th#1\relbar$\hss}%
	\lower#2\hbox{$\m@th#1\relbar$}%
}
\providecommand*{\rightrightarrowsfill@}{%
	\arrowfill@\relrelbar\relrelbar\rightrightarrows
}

\providecommand*{\xrightrightarrows}[2][]{%
	\ext@arrow 0359\rightrightarrowsfill@{#1}{#2}%
}
\begin{document}
	\title{Integral models of Harish-Chandra modules of the finite covering groups of PU(1,1)}
	\author{Takuma Hayashi\thanks{Graduate School of Mathematical Sciences, The University of Tokyo, 3-8-1 Komaba Meguro-ku Tokyo 153-8914, Japan, htakuma@ms.u-tokyo.ac.jp}}
	\date{}
	\maketitle
	\begin{abstract}
		We compute integral models of real and cohomological induction for finite covering groups of PU(1,1).
	\end{abstract}
	\section{Introduction}
	In representation theory of real reductive Lie groups over the complex numbers, the functor $I^{\fg,K}_{\fq,M}$ plays an important role. For example, let $G$ be a real reductive Lie group, and $(\fg,K)$ be its Harish-Chandra pair. If $(\fq,M)$ is a real parabolic subpair of $(\fg,K)$, $I^{\fg,K}_{\fq,M}$ is an algebraic analog of the real parabolic induction for $G$ (\cite{MR1330919} Proposition 11.47). If $(\fq,M)$ is a parabolic subpair associated to a $\theta$-stable parabolic subalgebra $\fq$, one can define cohomological induction and $A_\fq(\lambda)$-modules from $I^{\fg,K}_{\fq,M}$ (\cite{MR1330919} (5.3b), (5.6)).
	
	Fabian Januszewski introduced the functor $I^{\fg,K}_{\fq,M}$ over fields of characteristic 0 in \cite{MR3770183}. In \cite{MR4007195} and \cite{MR3853058}, the functor $I^{\fg,K}_{\fq,M}$ was constructed for maps $(\fq,M)\to(\fg,K)$ over an arbitrary commutative ring. We also gave abstract study of the flat base change theorem of the derived functor $\bR I^{\fg,K}_{\fq,M}$ in \cite{MR3853058}. It shows that $\bR I^{\fg,K}_{\fq,M}$ produces rational and integral models of $A_\fq(\lambda)$-modules under mild conditions. The rational models have applications to special values of automorphic $L$-functions and Rankin-Selberg $L$-functions (\cite{MR3937337}, \cite{1604.04253}). We should be able to apply integral models as well to strengthen the rationality results. The use of integral models for special values of $L$-functions was also suggested in \cite{MR3970997}. 
	
	Currently, no nontrivial computational results on $I^{\fg,K}_{\fq,M}$ over commutative rings are known. The aim of this paper is to compute some modules obtained by the functor $I^{\fg,K}_{\fq,M}$ over commutative rings explicitly. More specifically, we study the following three things in an easy setting:
	\begin{itemize}
		\item Compute $I^{\fg,K}_{\fq,M}$ when the flat base change theorem holds.
		\item See when an induced module is defined over a smaller ring.  
		\item Find an interesting counterexample to the flat base change theorem.
	\end{itemize}
	However, it appears difficult to compute $I^{\fg,K}_{\fq,M}$ directly from our construction in \cite{MR4007195}. The main issue is representation theory of $K$. In the theory of $(\fg,K)$-modules over the complex numbers, it is crucial that algebraic representations of compact Lie groups or corresponding complex reductive algebraic groups satisfy complete reducibility. Let $(\fg,K)$ be a Harish-Chandra pair over the complex numbers. Suppose that $K$ is reductive. Choose a maximal compact subgroup $K_\bR$ of $K$. The following results are consequences of complete reducibility:
	\begin{itemize}
		\item The coordinate ring of $K$ is cosemisimple (\cite{MR0252485} Chapter XIV and XV).
		\item Every $K$-module is both injective and projective (\cite{MR1330919} Lemma 2.4).
		\item Irreducible representations of $K$ form a family of projective generators of the category of $K$-modules.
		\item The $K_\bR$-finite distributions on $K_\bR$ form a convolution algebra $R(K_\bR)$ which is called the Hecke algebra. This algebra is approximately unital. The category of locally finite representations of $K_\bR$ is canonically isomorphic to that of approximately unital $R(K_\bR)$-modules. More generally, to $(\fg,K)$ is attached an approximately unital algebra $R(\fg,K)$ which is commonly referred to as the Hecke algebra. The category of $(\fg,K)$-modules is isomorphic to that of approximately unital $R(\fg,K)$-modules (\cite{MR1330919} I.4 Theorem). Thanks to this fact, $(\fg,K)$-modules are treated like modules over rings in \cite{MR1330919}. For instance, $I^{\fg,K}_{\fq,M}$ can be constructed as a Hom-type right adjoint functor between modules.
		\item We can define the dual Zuckerman functor $\Pi$ by a base-change type construction (see \cite{MR1330919} (2.8) and \cite{1604.04253} 1.4.2).
		\item We can construct the so-called standard projective and injective resolutions of $(\fg,K)$-modules from the Koszul complex (\cite{MR1330919} Theorem 2.122). It enables us to compute the derived functor modules explicitly.
	\end{itemize}
	The highest weight theory is also important for deeper study.  In this theory, weight space decomposition of representations of tori and the structure of root systems play a fundamental role.
	
	Though we cannot expect such results for reductive algebraic groups over fields of positive characteristics in general, representation theory of diagonalizable groups over an arbitrary commutative ring has quite similar features to that over the complex numbers. See \cite{MR2015057} I.2.11 for example. In Section 3, we reformulate it in terms of the Hecke algebra. We also introduce the relative Hecke algebra in Section 4. As an application, we explicitly compute integral models of complexified Harish-Chandra pairs of finite coverings of $\PU(1,1)$ in Section 5. There are many choices of integral models. If one assumes that an integral model is compatible with the real structure of the Harish-Chandra pair of coverings groups of $\PU(1,1)$, (a standard form of) the unitary group $\mathrm{U}(1)$ will be a typical example of $K$. In this paper, we adopt the split torus $T^1$ for a candidate of $K$ instead of $\mathrm{U}(1)$ since $\mathrm{U}(1)$ is not diagonalizable. Let $\bZ$ denote the ring of integers. The theories over $\bZ\left[1/2\right]$ for these two choices are related by the base change to $\bZ\left[1/2,\sqrt{-1}\right]$ (\cite{MR3853058} Variant 3.2.12, Variant 3.2.13). In Section 5, we introduce the notion of split integral forms of the Harish-Chandra pairs of finite covering groups of $\PU(1,1)$ which are Harish-Chandra pairs over $\bZ$ (Definition \ref{defn:splitpair}). We classify them (Theorem \ref{thm:classification}). We also construct models of parabolic subpairs. Then we compute the induction $I^{\fg,K}_{\fq,M}$ of integral modules which are free of rank 1 over $\bZ$. We prove that in our setting, similar results to the case of the complex numbers are obtained for models of $\theta$-stable Borel subalgebras, where $\theta$ is the Cartan involution (Theorem \ref{thm:thetastableinduction}). If we invert some integers, we also show that similar results to the case of the complex numbers hold for models of real parabolic subpairs (Proposition \ref{prop:compactpicture}). We also add $\sqrt{-1}$ to the base ring to study the conjugation of models of principal series representations. In particular, we give a criterion when the induced modules descend (Proposition \ref{prop:descent}, Variant \ref{var:descent1}, Variant \ref{var:descent2}). The main theorem is on certain integral models of real parabolic supairs:
	\begin{thm}[Theorem \ref{thm:partlyvanishing}, Variant \ref{var:partlyvanishing}]
		There exist $\bZ$-forms of irreducible highest and lowest weight representations of the finite covering groups of $\PU(1,1)$ which are obtained from principal series type induction.
	\end{thm}
	This result asserts that on the course of induction, some part of the real parabolic induction vanishes when we work over the integers. Roughly speaking, this happens because the singularity of $\mathrm{U}(1)$ at $2$ shifts to a kind of singularity in representation theory by avoiding it. A precise factor is that the Iwasawa decomposition fails in our setting in the absence of $1/2$. This theorem gives a nontrivial counterexample to the flat base change theorem when the condition (i) of \cite{MR3853058} Theorem 3.1.7 fails. At the end of this paper, we will remark that similar computations work in the contraction setting of Joseph Bernstein, Nigel Higson, and Eyal Subag (\cite{10.1093/imrn/rny147}, \cite{10.1093/imrn/rny146}). This gives a contraction family of principal series representations.
	\renewcommand{\abstractname}{Acknowledgments}
	\begin{abstract}
		The author is grateful to his advisor Professor Hisayosi Matumoto for helpful discussions, suggestions, and careful reading of drafts of this paper.
		
		Thanks also to Masatoshi Kitagawa for comments on principal series representations of covering groups of SU(1,1).
		
		He also thanks the referee for suggesting to study the conjugation action.
		
		This work was supported by JSPS Kakenhi Grant Number JP15J06457 and the Program for Leading Graduate Schools, MEXT, Japan.
	\end{abstract}
	\section{Notation}
	Let $\bZ_{>0}\subset\bZ$ be the set of positive integers. For a prime number (resp.\ a Gaussian prime) $\ell$, we denote the $\ell$-adic valuation of $\bZ$ (resp.\ $\bZ\left[\sqrt{-1}\right]$) by $\mathrm{ord}_\ell$. 
	
	The field of rational (resp.\ complex) numbers will be denoted by $\bQ$ (resp.\ $\bC$). For an integral domain $k$, let $\Frac(k)$ be the field of fractions of $k$. If $k$ is the polynomial ring $\bC\left[z\right]$, write $\bC(z)=\Frac\bC\left[z\right]$.
	
	For a commutative ring $k$, let $k^\times$ denote the group of units of $k$.
	
	For an object $X$ of a category, we refer to its identity map as $\id_X$.
	
	We fix a commutative ring $k$ in Section 2-4. For $k$-modules $V$ and $W$, let $\Hom_k(V,W)$ denote the $k$-module of $k$-homomorphisms $V\to W$. When $V=W$ (resp.\ $W=k$), we write $\Hom_k(V,V)=\End_k(V)$ (resp.\ $\Hom_k(V,k)=V^\ast$). Similar notations to modules or comodules over other algebraic objects will be used like $\Hom_{\fg,K}(V,W)$.
	
	For a $k$-algebra $\cA$, we denote the multiplication map $\cA\otimes_k\cA\to \cA$ by $m_\cA$.
	
	Let $C$ be a flat coalgebra over $k$. Let us denote the comultiplication and the counit of $C$ by $\Delta_C$ and $\epsilon_C$ respectively. For a (right) $C$-comodule $V$, let us denote the coaction $V\to V\otimes_k C$ on $V$ by $\rho_V$. We put the natural structure of a $k$-algebra on $C^\ast$ (see \cite{MR0252485} Proposition 1.1.1). The category of $C$-comodules (resp.\ left $C^\ast$-modules) will be denoted by $C\comod$ (resp.\ $C^\ast\cmod$). Recall that we have a natural functor
	\begin{equation}\label{eq:comodtorationalmod}
	C\comod\to C^\ast\cmod
	\end{equation}
	(\cite{MR0252485} Proposition 2.1.1).
	
	For a flat affine group scheme $K$ over $k$, we denote the category of representations of $K$ by $K\cmod$. For a map $M\to K$ of flat affine group schemes, we denote the right adjoint functor to the forgetful functor $K\cmod\to M\cmod$ by $\Ind^K_M$ (see \cite{MR2015057} I.3.3 and I.3.4).
	
	For a flat affine group scheme over $k$, its Lie algebra will be denoted by the corresponding German letter.
	
	Recall that we introduced the notions of Harish-Chandra pairs $(\cA,K)$, $(\cA,K)$-module, and their weak analogs over commutative rings in \cite{MR4007195} Section 2.1. In this paper, we remove their differential graded structures for simplicity (\cite{MR4007195} Remark 1.2.4). We extended the notions of Harish-Chandra pairs and their modules in \cite{MR3853058} Section 1.1. In this paper, we adopt the definitions in \cite{MR3853058} Section 1.1. Taking it into account, we denote the category of weak $(\cA,K)$-modules (resp.\ $(\cA,K)$-modules) by $(\cA,K)\cmod_w$ (resp.\ $(\cA,K)\cmod$). We write $\cJ_{\cA,K}$ for the canonical embedding $(\cA,K)\cmod\hookrightarrow(\cA,K)\cmod_w$. When $\cA$ is the enveloping algebra of a Lie algebra $\fg$, we write $(\fg,K)$, $(\fg,K)\cmod$ etc.
	
	Let $f=(f_a,f_g):(\fh,L)\to(\fg,K)$ be a map of Harish-Chandra pairs over $k$ in the sense of \cite{MR4007195} Section 2.1. Remark that this makes sense even for the refined definition of Harish-Chandra pairs. Let $\cF^{\fg,K}_{\fh,L}:(\fg,K)\cmod\to(\fh,L)\cmod$ be the forgetful functor (\cite{MR4007195} Proposition 2.2.1). We will denote the right adjoint functor of $\cF^{\fg,K}_{\fh,L}$ by $I^{\fg,K}_{\fh,L}$ (\cite{MR3853058} Lemma 1.1.9). If $L=K$ and $f_g=\id_K$, we write
	\[\ind^\fg_\fh\coloneqq P^{\fg,K}_{\fh,K}\]
	\[\pro^\fg_\fh\coloneqq I^{\fg,K}_{\fh,K}\]
	(see \cite{MR4007195} Theorem 1.2.2).
	
	For a Harish-Chandra pair $(\fg,K)$ over $k$ and a (flat) $k$-algebra $k'$, we will denote the base change to $k'$ (\cite{MR3853058} Proposition 3.1.1, Remark 3.1.2) by the same symbol $(\fg,K)$ in Section 4 and 5 if there is no risk of confusion.
	
	\section{The Hecke algebras of diagonalizable groups}
	Let $\Lambda$ be an additive group, $k\left[\Lambda\right]$ be its group algebra over $k$, and $\{t^\lambda\}_{\lambda\in\Lambda}$ be the standard basis of $k\left[\Lambda\right]$. Then $k\left[\Lambda\right]$ is a commutative and cocommutative Hopf algebra over $k$ for $\Delta_{k\left[\Lambda\right]}(t^\lambda)=t^\lambda\otimes t^\lambda$ and $\epsilon_{k\left[\Lambda\right]}(t^\lambda)=1$. Write $T=\Spec k\left[\Lambda\right]$. The flat affine group scheme $T$ is called a diagonalizable group over $k$ in \cite{MR0212024}.
	
	For each element $\lambda\in\Lambda$, denote the subcomodule $kt^\lambda $ of the regular comodule $k\left[\Lambda\right]$ by $k_\lambda$, and let $p_\lambda\in k\left[\Lambda\right]^\ast$ be the projection to the $\lambda$-component, namely,
	\[p_\lambda(t^{\lambda'})=\left\{ \begin{array}{ll}
	1& (\lambda=\lambda') \\
	0& (\lambda\neq\lambda').
	\end{array} \right.\]
	Under the canonical identification $k\left[\Lambda\right]^\ast\cong\prod_{\lambda\in\Lambda} kp_\lambda$, $k\left[\Lambda\right]^\ast$ is the product of copies of the algebra $k$ as a $k$-algebra. We will denote elements of $k\left[\Lambda\right]^\ast$ by formal sums $\sum c_\lambda p_\lambda$.
	
	Let $V$ be a $k\left[\Lambda\right]$-comodule, and $v\in V$. Write $\rho_V(v)=\sum v_\lambda\otimes t^\lambda$. Recall that $V$ is a $k\left[\Lambda\right]^\ast$-module by \eqref{eq:comodtorationalmod}. Unwinding the definitions, we then have $p_\lambda v=v_\lambda$.
	\begin{lem}[\cite{MR2015057} I.2.11]
		For $\lambda,\lambda'\in\Lambda$ we have
		\[\Hom_{k\left[\Lambda\right]}(k_\lambda,k_{\lambda'})=\left\{ \begin{array}{ll}
		k& (\lambda=\lambda') \\
		0& (\lambda\neq\lambda').
		\end{array} \right.\]
	\end{lem}
	\begin{proof}
		The case $\lambda=\lambda'$ is obvious. Suppose that $\lambda\neq \lambda'$. Then any $k\left[\Lambda\right]$-comodule homomorphism $f:k_\lambda\to k_{\lambda'}$ is zero since
		\[f(1)=f(p_\lambda\cdot 1)=p_\lambda f(1)=0.\]
	\end{proof}
	Let $R(T)$ be the $k$-submodule of $k\left[\Lambda\right]^\ast$ spanned by $\{p_\lambda:\lambda\in\Lambda\}$, i.e., $R(T)=\mathop{\oplus} kp_\lambda\subset k\left[\Lambda\right]^\ast$. This is a subalgebra of $k\left[\Lambda\right]^\ast$ which is possibly nonunital. We call $R(T)$ the Hecke algebra of $T$. We will denote the category of nonunital $R(T)$-modules by $R(T)\cmod$. Notice that there are canonical functors
	\begin{equation}\label{eq:comodtoHeckemod}
	k\left[\Lambda\right]\comod\to k\left[\Lambda\right]^\ast\cmod\to R(T)\cmod.
	\end{equation}
	We call a unital $k\left[\Lambda\right]^\ast$-module (resp.\ an $R(T)$-module) $V$ rational (resp.\ approximately unital) if for any $v\in V$ there exists a finite subset $I\subset\Lambda$ such that $\sum_{\lambda\in I} p_\lambda v=v$.
	Their categories will be denoted by $k\left[\Lambda\right]^\ast\cmod^{\rat}$ and $R(T)\cmod^{\un}$ respectively. The reader may see \cite{MR0252485} Section 2.1 for a standard treatment of rational modules for coalgebras over fields.
	\begin{thm}\label{thm:comodisappunimod}
		The functors \eqref{eq:comodtoHeckemod} restrict to isomorphisms of categories
		\[k\left[\Lambda\right]\comod\cong k\left[\Lambda\right]^\ast\cmod^{\rat}\cong R(T)\cmod^{\un}.\]
	\end{thm}
	\begin{proof}
		It is clear that rational $k\left[\Lambda\right]^\ast$-modules restrict to approximately unital $R(T)$-modules. Let $V$ be a $k\left[\Lambda\right]$-comodule, and $v\in V$. Then one can find a finite subset $I\subset \Lambda$ such that $\rho_V(v)=\sum_{\lambda\in I}p_\lambda v\otimes t^\lambda$.
		The counitality of the coaction implies $\sum_{\lambda\in I}p_\lambda v=v$. Therefore we have shown that the functors \eqref{eq:comodtoHeckemod} factor through $k\left[\Lambda\right]^\ast\cmod^{\rat}$ and $R(T)\cmod^{\un}$.
		
		Let $V$ be a rational $k\left[\Lambda\right]^\ast$-module or an approximately unital $R(T)$-module, and $v\in V$. Choose a finite subset $I\subset\Lambda$ such that $\sum_{\lambda\in I}p_\lambda v=v$. For an element $\lambda\in\Lambda\setminus I$, we have $p_\lambda v=p_\lambda(\sum_{\mu\in I}p_\mu v)
		=(\sum_{\mu\in I}p_\lambda p_\mu)v=0$.
		
		We now construct the inverses. If we are given an approximately unital $R(T)$-module $M$, it naturally extends to a rational $k\left[\Lambda\right]^\ast$-module by an essentially finite sum
		$(\sum c_\lambda p_\lambda)v=\sum c_\lambda p_\lambda v$.
		For a rational $k\left[\Lambda\right]^\ast$-module $V$, one can define the coaction map $\rho_V:V\to V\otimes k\left[\Lambda\right]$ by $\rho_V(v)=\sum_\lambda p_\lambda v\otimes t^\lambda$.
		These functors provide the desired inverses.
	\end{proof}
	\begin{var}
		Let $k$ be a $\bC$-algebra, and $K$ be a complex reductive group. In general, for a free $k$-module $W$ of finite rank, $\End_k W$ is a coalgebra over $k$ for the canonical isomorphism $\End_k(W)\cong\End_k(W)^\ast$. For each irreducible representation $V$ of $K$, we have a coalgebra homomorphism
		\[\End_k(V\otimes_\bC k)\to\cO(K\otimes_\bC k).\]
		Passing to all isomorphism classes, we get an isomorphism
		\[\oplus\End_k(V\otimes_\bC k')\cong\cO(K\otimes_\bC k')\]
		of coalgebras (the algebraic Peter-Weyl theorem). Passing to summand wise duals, we obtain an approximately unital ring $R(K\otimes_\bC k)\cong \oplus\End_k(V\otimes_\bC k)$ which is compatible with base changes. Moreover, the category of approximately unital $R(K\otimes_\bC k)$-modules is isomorphic to $K\otimes_\bC k\cmod$.
	\end{var}
	\begin{rem}[\cite{MR1843016}, \cite{MR2138086}]
		It is known that the isomorphism $C\comod\cong C^\ast\cmod^{\rat}$ is valid for coalgebras $C$ which are projective as $k$-modules. This is satisfied when $C$ is the coordinate ring of a split reductive group over $k$.
	\end{rem}
	\begin{cor}[$T$-weight module decomposition, \cite{MR2015057} I.2.11]
		Let $V$ be a $k\left[\Lambda\right]$-comodule. For each $\lambda\in\Lambda$, set $V_\lambda=p_\lambda(V)$. Then we have a decomposition $V=\oplus_{\lambda\in\Lambda}V_\lambda$ as a $k\left[\Lambda\right]$-comodule. We call $V_\lambda$ the $\lambda$-weight module.
	\end{cor}
	\begin{cor}\label{cor:projgenofCcomod}
		The comodules $k_\lambda$ form a family of projective generators of the category $k\left[\Lambda\right]\comod$.
	\end{cor}
	\begin{cor}
		The category $k\left[\Lambda\right]\comod$ has enough projectives. 
	\end{cor}
	\begin{proof}
		This follows from \cite{MR1291599} Proposition 4.6.6.
	\end{proof}
	\begin{cor}
		A $k\left[\Lambda\right]$-comodule $V$ is injective (resp.\ projective) if and only if each weight module $V_\lambda$ is injective (resp.\ projective) as a $k$-module.
	\end{cor}
	\begin{proof}
		Thanks to the action of $R(T)$, each of $V_\lambda$ is injective (resp.\ projective) in $k\left[\Lambda\right]\comod$ if and only if it is so in the category of $k$-modules. The assertion now follows since each $V_\lambda$ is a retract of $V$.
	\end{proof}
	As an application we can introduce the notion of $T$-finite part:
	\begin{cor}
		\begin{enumerate}
			\renewcommand{\labelenumi}{(\arabic{enumi})}
			\item The category $R(T)\cmod^{\un}$ is an abelian subcategory of $R(T)\cmod$.
			\item The embedding
			\begin{equation}\label{eq:embedding}
			k\left[\Lambda\right]\comod\cong R(T)\cmod^{\un}\hookrightarrow R(T)\cmod
			\end{equation}
			admits an exact right adjoint functor $(-)_T$.
		\end{enumerate}
	\end{cor}
	\begin{proof}
		Part (1) follows since the functors of \eqref{eq:embedding} commute with the forgetful functors to the category of $k$-modules which are exact and conservative.
		
		To prove (2), let $V$ be an $R(T)$-module. We say that an element $v\in V$ is $T$-finite if there is a finite subset $I\subset\Lambda$ such that $\sum_{\lambda\in I}p_\lambda v=v$. We denote the subset of $T$-finite elements of $V$ by $V_T\subset V$. Then $V_T$ is an approximately unital $R(T)$-submodule of $V$. Moreover $V_T\subset V$ exhibits an $R(T)\cmod^{\un}$-colocalization of $V$. It is proved in a similar way to \cite{MR1330919} Proposition 1.55 that the resulting colocalization functor is exact.
	\end{proof}
	\begin{cor}
		We have $(k\left[\Lambda\right]^\ast)_T=R(T)$.
	\end{cor}
	\begin{rem}
		The arguments above work if we replace $k\left[\Lambda\right]$ by a diagonal coalgebra $C$ in the sense of \cite{anel2013sweedler} Example 1.3.7. It is equivalent to saying that there are free bases $\{t^\lambda\}$ and $\{s^\lambda\}$ such that $\Delta_C(s^\lambda)=t^\lambda\otimes s^\lambda$. In fact, the coassociativity of $\Delta_C$ implies $\Delta_C(t^\lambda)\otimes s^\lambda=t^\lambda\otimes t^\lambda\otimes s^\lambda$. In particular, we have $\Delta_C(t^\lambda)=t^\lambda\otimes t^\lambda$.
	\end{rem}
	Thanks to Corollary \ref{cor:projgenofCcomod}, the projective model structure also exists.
	\begin{note}
		Let $\cA$ be an abelian category, $M$ be an object, and $n$ be an integer. Then we denote the cochain complexes
		\[\cdots\to 0\to\overset{-n}{M}=\overset{-n+1}{M}\to 0\to\cdots\]
		\[\cdots\to 0\to\overset{-n}{M}\to 0 \to\cdots\]
		by $D^nM$ and $S^nM$ respectively. Notice that we have a natural inclusion $S^{n-1}M\to D^nM$.
	\end{note}
	\begin{cor}
		There exists a combinatorial model structure on the category of cochain complexes of $k\left[\Lambda\right]$-comodules which is described as follows:
		\begin{enumerate}
			\item[(F)] A map is a fibration if and only if it is an epimorphism.
			\item[(W)] A map is a weak equivalence if and only if it is a quasi-isomorphism.
			\item[(C)] A map is a cofibration if and only if it is a degreewise split monomorphism with a cofibrant cokernel.
		\end{enumerate}
		Moreover, the generating cofibrations (resp.\ trivial cofibrations) are the standard embeddings $S^{p-1}k_\lambda\to D^pk_\lambda$ (resp.\ $0\to D^pk_\lambda$), where $p$ runs through all integers.
	\end{cor}
	\begin{proof}
		This is a direct consequence of \cite{MR1912401} Theorem 5.7: for a locally presentable abelian category $\cA$ equipped with a small set $G$ of projective generators, there exists a model structure on the category of cochain complexes of objects of $\cA$ such that the following conditions are satisfied:
		\begin{enumerate}
			\item[(F)] A morphism is a fibration if and only if it is an epimorphism.
			\item[(W)] A morphism is a weak equivalence if and only if it is a quasi-isomorphism.
			\item[(C)] A morphism is a cofibration if and only if it is a degreewise split monomorphism with a cofibrant cokernel.
		\end{enumerate}
		Moreover, the generating cofibrations (resp.\ trivial cofibrations) are the standard embeddings $S^{m-1}P\to D^mP$ (resp.\ $0\to D^mP$), where $m$ runs through all integers, and $P$ are the members of $G$.
	\end{proof}
	\section{The Hecke algebra in the relative setting}
	Let $T=\Spec k\left[\Lambda\right]$ be a diagonalizable group as in the previous section. 
	\begin{lem}\label{lem:closedmonoidalstructure}
		Let $V$ and $V'$ be $T$-modules.
		\begin{enumerate}
			\renewcommand{\labelenumi}{(\arabic{enumi})}
			\item The action of $R(T)$ corresponding to the tensor representation $V\otimes_k V'$ of $T$ is given by
			\[p_\lambda(v\otimes v')=\sum_\mu p_\mu v\otimes p_{\lambda-\mu}v'.\]
			\item The Hom $k$-module $\Hom_k(V,V')$ is an $R(T)$-module for
			\[(p_\lambda f)(v)=\sum_{\mu\in \Lambda} p_{\lambda+\mu} f(p_\mu v).\]
			Moreover, the $T$-finite part $\Hom_k(V,V')_T$ exhibits the closed structure of the symmetric monoidal category $T\cmod$.
		\end{enumerate}
	\end{lem}
	\begin{proof}
		Part (1) follows from
		\[\rho_{V\otimes_k V'}(v\otimes v')=\sum_{\lambda,\lambda'\in\Lambda} p_\lambda v\otimes p_{\lambda'}v'\otimes t^{\lambda+\lambda'}
		=\sum_{\lambda\in\Lambda}\sum_{\mu\in\Lambda}p_\mu v\otimes p_{\lambda-\mu}v'\otimes t^\lambda.\]
		To see (2), let $\lambda,\lambda'\in\Lambda$, $f\in\Hom_k(V,V')$, and $v\in V$. Then we have
		\[\begin{split}
		(p_\lambda(p_{\lambda'} f))(v)
		&=\sum_{\mu\in\Lambda}p_{\lambda+\mu}(p_{\lambda'} f)(p_\mu v)\\
		&=\sum_{\mu\in\Lambda}p_{\lambda+\mu}
		\sum_{\mu'\in\Lambda}p_{\lambda'+\mu'}f(p_{\mu'}p_\mu v)\\
		&=\sum_{\mu\in\Lambda}p_{\lambda+\mu}p_{\lambda'+\mu} f(p_\mu v)\\
		&=\left\{ \begin{array}{ll}
		\sum p_{\lambda+\mu} f(p_\mu v)& (\lambda=\lambda') \\
		0& (\lambda\neq\lambda'),
		\end{array} \right.\\
		&=((p_\lambda p_{\lambda'})f)(v).
		\end{split}\]
		The internal structure of the category of $k$-modules induces an internal structure on $\Hom_k(-,-)_T$.
	\end{proof}
	Put the natural structure of a $k\left[\Lambda\right]$-module on $k\left[\Lambda\right]^\ast$. Explicitly, the action is given by
	\[t^\lambda(\sum_\mu c_\mu p_\mu)=\sum_\mu c_\mu p_{\mu-\lambda}.\]
	Then $R(T)$ is a $k\left[\Lambda\right]$-submodule of $k\left[\Lambda\right]^\ast$. Therefore we have a natural isomorphism $R(T)\otimes_k V\cong R(T)\otimes_{k\left[\Lambda\right]} (k\left[\Lambda\right]\otimes_k V)$ for any $k$-module $V$. If we write $\lambda\in\Lambda$ and $v=\sum_\mu t^\mu\otimes v_\mu \in k\left[\Lambda\right]\otimes V$, the inverse image of $p_\lambda\otimes v$ is $\sum_\mu p_{\lambda-\mu}\otimes v_\mu$. 
	
	Suppose next that $V$ is a $T$-module. By definition of a $T$-module, $\id_{k\left[\Lambda\right]}$ induces a $k\left[\Lambda\right]$-module automorphism $V\otimes_k k\left[\Lambda\right]\cong V\otimes_k k\left[\Lambda\right]$. Taking the base change $R(T)\otimes_{k\left[\Lambda\right]}-$ and switching the factors, we get an isomorphism of $k$-modules
	\[\tau_V:R(T)\otimes_k V\cong V\otimes_k R(T);\]
	\[p_\lambda\otimes v\mapsto \sum_\mu p_\mu v\otimes p_{\lambda-\mu}.\]
	Its inverse is given by $v\otimes p_\lambda\mapsto \sum_\mu p_{\lambda+\mu}\otimes p_\mu v$.
	
	Let $(\cA,T)$ be a weak Harish-Chandra pair.
	\begin{lem}\label{lem:weakrelativeHeckalgebra}
		The product
		\[\begin{split}
		(\cA\otimes_k R(T))\otimes_k (\cA\otimes_k R(T))
		&\xrightarrow{\cA\otimes_k\tau_\cA\otimes_k R(T)}\cA\otimes_k \cA\otimes_k R(T)\otimes_k R(T)\\
		&\xrightarrow{m_\cA\otimes m_{R(T)}} \cA\otimes_k R(T);
		\end{split}\]
		\[(a\otimes p_\lambda)\otimes(b\otimes p_\mu)\mapsto ap_{\lambda-\mu} b\otimes p_\mu\]
		is associative. We will refer to the resulting associative algebra as $\cA\sharp R(T)$.
	\end{lem}
	\begin{proof}
		Lemma \ref{lem:closedmonoidalstructure} (1) implies that for $a\otimes p_\lambda, b\otimes p_\mu,c\otimes p_\nu\in\cA\otimes_k R(T)$, we have
		\[\begin{split}
		((a\otimes p_\lambda)(b\otimes p_\mu))(c\otimes p_\nu)
		&=(ap_{\lambda-\mu}b\otimes p_\mu)(c\otimes p_\nu)\\
		&=ap_{\lambda-\mu}bp_{\mu-\nu}c\otimes p_\nu\\
		&=ap_{\lambda-w}(bp_{\mu-\nu}c)\otimes p_\nu\\
		&=(a\otimes p_\lambda)(bp_{\mu-\nu}c\otimes p_\nu)\\
		&=(a\otimes p_\lambda)((b\otimes p_\mu)(c\otimes p_\nu)).
		\end{split}\]
	\end{proof}
	We say that an $\cA\sharp R(T)$-module is approximately unital if it is so as an $R(T)$-module. The category of approximately unital $\cA\sharp R(T)$-modules will be denoted by $\cA\sharp R(T)\cmod^{\un}$.
	\begin{cor}\label{cor:fundamentaltheoremofweakHeckealgebra}
		There is an isomorphism $(\cA,T)\cmod_w\cong \cA\sharp R(T)\cmod^{\un}$.
	\end{cor}
	\begin{proof}
		Similar computations to the proof of Lemma \ref{lem:weakrelativeHeckalgebra} show that a weak $(\cA,T)$-module $M$ is an approximately unital $\cA\sharp R(T)$-module for
		\[(a\otimes p_\lambda)m=a(p_\lambda m).\]
		Conversely, if we are given an approximately unital $\cA\sharp R(T)$-module $M$ then define an action of $\cA$ on $M$ by the essentially finite sum
		\[am=\sum_{\lambda\in\Lambda} (a\otimes p_\lambda) m.\]
		These correspondences determine the desired isomorphism.
	\end{proof}
	\begin{rem}
		The functor $(-)_T$ is compatible with the action of $\cA$. Namely, the embedding of $\cA\sharp R(T)\cmod^{\un}$ to the category $\cA\sharp R(T)\cmod$ of $\cA\sharp R(T)$-modules admits an exact right adjoint functor $(-)_T$ which fits into the commutative diagram
		\[\xymatrix{\cA\sharp R(T)\cmod\ar[r]^{(-)_T}\ar[d]&\cA\sharp R(T)\cmod^{\un}\ar[d]\\
			R(T)\cmod\ar[r]_{(-)_T}&R(T)\cmod^{\un},
		}\]
		where the vertical arrows are obtained by the restriction along $R(T)\to \cA\sharp R(T)$.
	\end{rem}
	We next consider its version for Harish-Chandra pairs $(\cA,T)$. Assume that $T$ is the split torus $T^n$ of rank $n\geq 0$, i.e., $T=\Spec k\left[\bZ^n\right]$. The Lie algebra $\ft=\ft^n$ has a basis $\{H_1,\cdots, H_n\}$ which acts on $R(T^n)$ by $H_ip_\lambda=\lambda_ip_\lambda$ for an integer $1\leq i\leq n$ and $\lambda=(\lambda_j)\in\bZ^n$.	Then the algebra structure of $\cA\sharp R(T)$ descends to $\cA\otimes_{U(\ft^n)}R(T^n)$. We will refer to it as $R(\cA,T^n)$. An $R(\cA,T^n)$-module is said to be approximately unital if it is so as an $R(T^n)$-module.
	\begin{cor}\label{cor:relativeHecke}
		Corollary \ref{cor:fundamentaltheoremofweakHeckealgebra} induces an isomorphism of categories of approximately unital $R(\cA,T^n)$-modules and $(\cA,T^n)$-modules.
	\end{cor}
	\begin{rem}
		It is straightforward to prove the differential graded analogs of Lemma \ref{lem:weakrelativeHeckalgebra} - Corollary \ref{cor:relativeHecke}.
	\end{rem}
	\section{Integral models of representations of split semisimple Lie groups of type $A_1$}
	Fix a positive integer $n>0$. Then the Harish-Chandra pair over $\bC$ associated to the $n$-cover of $\PU(1,1)$ is given as follows:
	\[\fsl_2=\left\{\left(\begin{array}{cc}
	a&b\\
	c&-a\\
	\end{array}
	\right):~a,b,c\in \bC\right\}\]
	\[T^1=\Spec \bC\left[t^{\pm 1}\right]\]
	\[\ft^1\to\fsl_2;~H_1\mapsto\frac{n}{2}\left(\begin{array}{cc}
	1&0\\
	0&-1\\
	\end{array}
	\right)\]
	\[t\cdot \left(\begin{array}{cc}
	0&1\\
	0&0\\
	\end{array}
	\right)=t^n\left(\begin{array}{cc}
	0&1\\
	0&0\\
	\end{array}
	\right)\]
	\[t\cdot \left(\begin{array}{cc}
	0&0\\
	1&0\\
	\end{array}
	\right)=t^{-n}\left(\begin{array}{cc}
	0&0\\
	1&0\\
	\end{array}
	\right)\]
	\[t\cdot \left(\begin{array}{cc}
	1&0\\
	0&-1\\
	\end{array}
	\right)=\left(\begin{array}{cc}
	1&0\\
	0&-1\\
	\end{array}
	\right),\]
	where $t\in T^1$. The Cartan involution $\theta$ on $\fsl_2$ is given by
	\[\theta\left(\left(\begin{array}{cc}
	0&1\\
	0&0\\
	\end{array}
	\right)\right)=-\left(\begin{array}{cc}
	0&1\\
	0&0\\
	\end{array}
	\right)\]
	\[\theta\left(\left(\begin{array}{cc}
	0&0\\
	1&0\\
	\end{array}
	\right)\right)=-\left(\begin{array}{cc}
	0&0\\
	1&0\\
	\end{array}
	\right)\]
	\[\theta\left(\left(\begin{array}{cc}
	1&0\\
	0&-1\\
	\end{array}
	\right)\right)=\left(\begin{array}{cc}
	1&0\\
	0&-1\\
	\end{array}
	\right).\]
	
	The subpair associated to a minimal parabolic subgroup of the $n$-cover of PU(1,1) is given by
	\[\fq_\bC=\left\{\left(\begin{array}{cc}
	a&-a+b\\
	a+b&-a\\
	\end{array}
	\right):~a,b\in\bC\right\}\]
	\[M_\bC=\Spec\bC\left[t^{\pm 1}\right]/(t^n-1)=\Ker (T^1\to T^1;~t\mapsto t^n).\]
	We also define $\theta$-stable Borel subalgebras $\fb_\bC$ and $\bar{\fb}_\bC$ by
	\[\fb_\bC=\left\{\left(\begin{array}{cc}
	a&b\\
	0&-a\\
	\end{array}
	\right):~a,b\in\bC\right\}\]
	\[\bar{\fb}_\bC=\left\{\left(\begin{array}{cc}
	a&0\\
	b&-a\\
	\end{array}
	\right):~a,b\in\bC\right\}.\]
	\begin{defn}\label{defn:splitpair}
		A split $\bZ$-form of $(\fsl_2,T^1)$ is a Harish-Chandra pair $(\fg,T^1)$ over $\bZ$, equipped with a $T^1$-equivariant Lie algebra homomorphism $\alpha:\fg\to\fsl_2$ such that the following conditions are satisfied:
		\begin{enumerate}
			\renewcommand{\labelenumi}{(\roman{enumi})}
			\item $\fg$ is free of finite rank as a $\bZ$-module.
			\item The $\bC$-linear extension $\fg\otimes_{\bZ}\bC\to\fsl_2$ is an isomorphism.
			\item The given map $\psi:\ft^1\to\fg$ is one-to-one onto the weight 0 module $\fg_0$.
			\item We have $\alpha(\psi(H_1))=\frac{n}{2}\left(\begin{array}{cc}
			1&0\\
			0&-1\\
			\end{array}
			\right)$.
		\end{enumerate}
		In particular, $\alpha$ gives rise to an isomorphism $(\fg\otimes_{\bZ}\bC,T^1)\cong(\fsl_2,T^1)$. A morphism $f:(\fg,T^1,\alpha)\to (\fg',T^1,\alpha')$ is a map $T^1$-equivariant Lie algebra homomorphism $f:\fg\to \fg'$ such that $\alpha'\circ f=\alpha$. Remark that for a morphism $f:(\fg,T^1,\alpha)\to (\fg',T^1,\alpha')$, $(f,\id_{T^1})$ is a map from $(\fg,T^1)$ to $(\fg',T^1)$.
	\end{defn}
	\begin{thm}[Classification]\label{thm:classification}
		For a positive integer $m$ and $c\in\bC^\times$, define a split $\bZ$-form $(\fg_{n,m},T^1,\alpha_c)$ as follows:
		\begin{itemize}
			\item $\fg_{n,m}$ has a free $\bZ$-basis $\{E,F,H\}$;
			\item The Lie bracket of $\fg_{n,m}$ is defined by
			\[\left[H,E\right]=nE\]
			\[\left[H,F\right]=-nF\]
			\[\left[E,F\right]=mH;\]
			\item The split torus $T^1=\Spec\bZ\left[t^{\pm 1}\right]$ acts on $\fg_{n,m}$ by
			\[t\cdot E=t^nE\]
			\[t\cdot F=t^{-n}F\]
			\[t\cdot H=H\]
			for $t\in T^1$;
			\item The $T^1$-equivariant Lie algebra homomorphism $\ft^1\to\fg_{n,m}$ sends $H_1\to H$.
			\item The realization homomorphism $\alpha_c:\fg_{n,m}\to\fsl_2$ is defined as
			\[\alpha_c(E)=c\left(\begin{array}{cc}
			0&1\\
			0&0\\
			\end{array}
			\right)\]
			\[\alpha_c(F)=\frac{mn}{2c}\left(\begin{array}{cc}
			0&0\\
			1&0\\
			\end{array}
			\right)\]
			\[\alpha_c(H)=\frac{n}{2}\left(\begin{array}{cc}
			1&0\\
			0&-1\\
			\end{array}
			\right)\]
		\end{itemize}
		This gives rise to a bijection between the set of isomorphism classes of split $\bZ$-forms and $\bZ_{>0}\times\bC^\times/\{\pm 1\}$.
	\end{thm}
	\begin{proof}
		Let $(m,c)\in \bZ_{>0}\times\bC^\times$. Then the map
		\[E\mapsto -E\]
		\[F\mapsto -F\]
		\[H\mapsto H\]
		determines an isomorphism $(\fg_{n,m},T^1,\alpha_c)\cong (\fg_{n,m},T^1,\alpha_{-c})$.
		
		Suppose that we are given a split $\bZ$-form $(\fg,T^1,\alpha)$ of $(\fsl_2,T^1)$. Set $H=\psi(H_1)\in\fg$. Since $\alpha$ is $T^1$-equivariant, we have a weight decomposition
		\[\fg=\fg_{-n}\oplus\fg_0\oplus\fg_{n},\]
		where each $T^1$-weight module is free of rank 1. Hence we can find a nonzero complex number $c$ which is unique up to sign and a unique element $E\in\fg$ such that
		\[\alpha(\fg_n)=\bZ\left(\begin{array}{cc}
		0&c\\
		0&0\\
		\end{array}
		\right)\]
		\[\alpha(E)=c\left(\begin{array}{cc}
		0&1\\
		0&0\\
		\end{array}
		\right).\]
		There is a unique positive integer $m$ such that $\left[\fg_n,\fg_{-n}\right]=m\bZ H$. We can then find a unique element $F\in\fg_{-n}$ such that $\left[E,F\right]=mH$. Since $\alpha$ is a Lie algebra homomorphism, we have
		\[\alpha(F)=\frac{mn}{2c}\left(\begin{array}{cc}
		0&0\\
		1&0\\
		\end{array}
		\right).\]
		Hence we obtain $(\fg,T^1,\alpha)\cong (\fg_{n,m},T^1,\alpha_c)$.
		
		To see that this correspondence is injective up to sign of $c$, take
		\[(m,c),(m',c')\in \bZ_{>0}\times\bC^\times.\]
		Suppose that there is an isomorphism \[f:(\fg_{n,m},T^1,\alpha_c)\cong(\fg_{n,m'},T^1,\alpha_{c'}).\]
		Since $f$ respects $\alpha$ and is $T^1$-equivariant, we have
		\[f(E)=\frac{c}{c'} E\]
		\[f(F)=\frac{c'}{c}F\]
		\[f(H)=H.\]
		Since each weight submodule of $\fg_{n,m}$ and $\fg_{n,m'}$ is free $\bZ$-module of rank 1, $\frac{c}{c'}\in\{\pm 1\}$ and $m=m'$. This completes the proof.
	\end{proof}
	Henceforth fix a $\bZ$-form $(\fg_{n,m},T^1,\alpha)$. Then the maximal integral models $\fb$ and $\bar{\fb}$ of $\fb_\bC$ and $\bar{\fb}_\bC$ are given by
	\[\fb=\alpha^{-1}(\fb_\bC)=\bZ E\oplus \bZ H\]
	\[\bar{\fb}=\alpha^{-1}(\bar{\fb}_\bC)=\bZ F\oplus \bZ H.\]
	In particular, these are independent of the choice of the parameter $c\in\bC^\times/\{\pm 1\}$.
	\begin{thm}\label{thm:thetastableinduction}
		Let $k$ be a commutative ring, and $\lambda\in\bZ$. Regard the $T^1$-module $k_\lambda$ as modules over $(\fb,T^1)$ and $(\bar{\fb},T^1)$ for the projections $\fb\to\ft^1$ and $\bar{\fb}\to\ft^1$. Then we have the following descriptions:
		\[\ind^{\fg_{n,m}}_{\bar{\fb}}(k_\lambda)=\oplus_{p\geq0} ky_{\lambda+np}\]
		\[\pro^{\fg_{n,m}}_\fb(k_\lambda)=\oplus_{p\geq 0}ky^{\lambda+np}\]
		\[y_{\lambda-n}=0\]
		\[Ey_{\lambda+np}=y_{\lambda+n(p+1)};\]
		\[Fy_{\lambda+np}=-\frac{1}{2}mp(np-n+2\lambda)y_{\lambda+n(p-1)};\]
		\[Hy_{\lambda+np}=(\lambda+np)y_{\lambda+np}\]
		\[t\cdot y_{\lambda+np}=t^{\lambda+np}y_{\lambda+np};\]
		\[y^{\lambda-n}=0;\]
		\[Ey^{\lambda+np}=-\frac{1}{2}m(p+1)(np+2\lambda)z^{\lambda+n(p+1)};\]
		\[Fy^{\lambda+np}=z^{\lambda+n(p-1)};\]
		\[Hy^{\lambda+np}=(\lambda+np)z^{\lambda+np};\]
		\[t\cdot y^{\lambda+np}=t^{\lambda+np}y^{\lambda+np},\]
		where $t\in T^1$.
	\end{thm}
	\begin{proof}
		Put $y_{\lambda+np}=E^p\otimes 1$, and define
		\[y^{\lambda+np}\in\pro^{\fg_{n,m}}_{\fb}(k_\lambda)\cong
		\Hom_{U(\fb)}(U(\fg_{n,m}),k_\lambda)_{T^1}\]
		by
		\[y^{\lambda+np}(F^q)=\left\{\begin{array}{ll}
		1& (p=q)\\
		0&(q\neq p).
		\end{array} \right.\]
		They form free bases. The actions $Fy_{\lambda+np}$ and $Ey^{\lambda+np}$ are computed as
		\[\begin{split}
		(Ey^{\lambda+np})(F^{p+1})
		&=y^{\lambda+np}(F^{p+1}E)\\
		&=y^{\lambda+np}(EF^{p+1}-\frac{1}{2}mnp(p+1)F^p-m(p+1)HF^p)\\
		&=-\frac{1}{2}mnp(p+1)-m(p+1)\lambda
		\end{split}\]
		\[\begin{split}
		Fy_{\lambda+np}
		&=FE^p\otimes 1\\
		&=E^pF\otimes 1-\frac{1}{2}mp(n(p+1)-2n+2\lambda)E^{p-1}\otimes 1\\
		&=-\frac{1}{2}mp(np-n+2\lambda)y_{\lambda+n(p-1)}.
		\end{split}\]
		The rest is obvious by definition.
	\end{proof}
	Note that $\ind^{\fg_{n,m}}_{\bar{\fb}}(k_\lambda)$ is finitely generated as a $U(\fg_{n,m})$-module by definition, and $\pro^{\fg_{n,m}}_\fb(k_\lambda)$ is not. However, they still have the same $T^1$-weights. Moreover, each of them is free of finite rank 1 as a $k$-module. In particular, these representations are admissible in the sense of \cite{10.1093/imrn/rny147}.
	\begin{rem}
		Suppose that $k$ is Noetherian. Since we are working with possibly torsion modules, it might be convenient to emphasize that $\ind^{\fg_{n,m}}_{\bar{\fb}}(k_\lambda)$ and $\pro^{\fg_{n,m}}_\fb(k_\lambda)$ satisfy the following condition on $T^1$-modules $V$ as an estimation of the size of $V$: for any finitely generated $T^1$-module $Q$, the $k$-module $\Hom_{T^1}(Q,V)$ is finitely generated. This condition is quite delicate. Suppose that $k$ is an integral domain. A finitely generated and torsion-free $k$-form of an admissible $(\fg\otimes_k \Frac(k),T^1)$-module is not admissible in general. In fact, we have a simple counterexample. Put $k=\bZ$, and consider a Harish-Chandra pair $(\bZ,\Spec\bZ)$. Set $V=\bZ\left[1/2\right]$, and put an action of the polynomial ring $\bZ\left[x\right]$ on $V$ by $x=2^{-1}$. Then $V$ is a finitely generated $\bZ\left[x\right]$-module, $V\otimes\bQ$ is of finite dimension over $\bQ$, and torsion free as a $\bZ$-module. However, $V$ is not finitely generated over $\bZ$. See also \cite{MR3853058} Proposition 4.1.2.
	\end{rem}
	\begin{rem}
		Let $(\fg,K)$ be a Harish-Chandra pair over a Noetherian domain $k$. Suppose that $\fg$ is finitely generated as a $k$-module. Then every finite subset of a $(\fg,K)$-module is contained in a finitely generated $(\fg,K)$-submodule by a similar argument to \cite{MR3853058} Proposition 3.2.4. In particular, if we are given a $(\fg,K)$-module $V$ such that $V\otimes_k \Frac (k)$ is irreducible as a $(\fg\otimes_k\Frac (k),K\otimes_k\Frac (k))$-module then there is a finitely generated $(\fg,K)$-submodule $V'\subset V$ such that $V'\otimes_k\Frac(k)\cong V\otimes_k\Frac(k)$.
	\end{rem}
	We next consider models of the real parabolic induction for $(\fg_{n,m},T^1,\alpha_{\frac{1}{2}})$. Set
	\[X=-2mnE+F+2mH\]
	\[Y=2mnE+F\]
	\[\fq=\bZ X\oplus\bZ Y\]
	\[M=\Spec\bZ\left[t\right]/(t^n-1)\]
	to obtain a subpair $(\fq,M)\subset(\fg_{n,m},T^1)$ over $\bZ$. Remark that if we write $\alpha_{\frac{1}{2}}\otimes_{\bZ}\bZ\left[1/2mn\right]$ for the scalar extension of $\alpha_{\frac{1}{2}}$ to $\bZ\left[1/2mn\right]$, we have
	\[(\alpha_{\frac{1}{2}}\otimes_{\bZ}\bZ\left[1/2mn\right])^{-1}(\fq_\bC)=
	\bZ\left[1/2mn\right] X\oplus \bZ\left[1/2mn\right] Y.\]
	
	Firstly, we discuss over a $\bZ\left[1/2mn\right]$-algebra $k$. 
	Regard $(\fg_{n,m},T^1)$ and $(\fq,M)$ as Harish-Chandra pairs over $k$.
	\begin{prop}\label{prop:compactpicture}
		The diagram
		\[\xymatrix{
			(\fq,M)\cmod\ar[r]^{I^{\fg_{n,m},T^1}_{\fq,M}}
			\ar[d]_{\cF^{\fm,M}_{\fq,M}}
			&(\fg_{n,m},T^1)\cmod\ar[d]^{\cF^{\ft^1,T^1}_{\fg_{n,m},T^1}}\\
			M\cmod\ar[r]_{\Ind^{T^1}_M}&T^1\cmod
		}\]
		is 2-commutative.
	\end{prop}
	\begin{proof}
		Observe that the summation of $\fq\subset\fg_{n,m}$ and $\psi:\ft^1\to\fg_{n,m}$ determines an isomorphism
		\begin{equation}\label{eq:Iwasawa}
		\fq\oplus\ft^1\cong\fg_{n,m}
		\end{equation}
		\[(-\frac{1}{4mn}X+\frac{1}{4mn}Y,\frac{1}{2n}H)\gets E\]
		\[(\frac{1}{2}X+\frac{1}{2}Y,-mH)\gets F\]
		\[H\gets H.\]
		Let $W$ be a $(\fq,M)$-module, and $\epsilon:I^{\fg_{n,m},T^1}_{\fq,M}(W)\to W$ be the counit. For a $T^1$-module $\chi$, we have
		\[\begin{split}
		\Hom_{T^1}(\chi,I^{\fg_{n,m},T^1}_{\fq,M}(W))
		&\cong\Hom_{\fg_{n,m},T^1}(\ind^{\fg_{n,m}}_{\ft^1} \chi,I^{\fg_{n,m},T^1}_{\fq,M}(W))\\
		&\cong\Hom_{\fq,M}({\rm ind}^{\fg_{n,m}}_{\ft^1} \chi,W)\\
		&\cong\Hom_M(\chi,W);
		\end{split}\]
		\[f\mapsto \epsilon\circ f.\]
		This proves the assertion.
	\end{proof}
	\begin{cor}
		The functor $I^{\fg_{n,m},T^1}_{\fq,M}$ is exact.
	\end{cor}
	\begin{proof}
		The left exactness follows since $I^{\fg_{n,m},T^1}_{\fq,M}$ is right adjoint. We show that it is also right exact. In view of Proposition \ref{cor:projgenofCcomod} and Proposition \ref{prop:compactpicture}, it will suffice to solve the lifting problem
		\[\xymatrix{&\Ind^{T^1}_M X\ar[d]^{\Ind^{T^1}_M(p)}\\
			k_\lambda\ar[r]\ar@{-->}[ru]&\Ind^{T^1}_MY,}\]
		where $p:X\to Y$ is a surjection of $M$-modules, and $\lambda\in\bZ$. This is equivalent to the following problem:
		\[\xymatrix{&X\ar[d]^p\\
			k_\lambda\ar[r]\ar@{-->}[ru]&Y.}\]
		The solution now exists from Proposition \ref{cor:projgenofCcomod}.
	\end{proof}
	\begin{cor}
		There is an isomorphism $\Ind^{T^1}_M\cong R(T^1)\otimes_{R(M)}-$. In particular, we obtain an isomorphism
		\[I^{\fg_{n,m},T^1}_{\fq,M}(-)\otimes_k k'\cong I^{\fg_{n,m}\otimes_k k',T^1\otimes_k k'}_{\fq\otimes_k k',M\otimes_k k'}(-\otimes_k k')\]
		for every homomorphism $k\to k'$ of $\bZ\left[1/2mn\right]$-algebras.
	\end{cor}
	The adjoint functor theorem also implies the following property:
	\begin{cor}
		The functor $I^{\fg_{n,m},T^1}_{\fq,M}$ admits a right adjoint functor.
	\end{cor}
	For $\epsilon\in\{0,\frac{1}{n},\frac{2}{n},\cdots,\frac{n-1}{n}\}$ and $\mu\in k$, define the structure of a $(\fq,M)$-module on $k$ by
	\[(-2mnE+F+2mH)\cdot 1=0;\]
	\[(2mnE+F)\cdot 1=\mu;\]
	\[t\cdot 1= t^{n\epsilon},\]
	where $t\in M$. We refer to it as $k_{n\epsilon,\mu}$.
	\begin{thm}[Fractional models of principal series representations]\label{thm:fractionalmodel}
		The $(\fg_{n,m},T^1)$-module $I^{\fg_{n,m},T^1}_{\fq,M}(k_{n\epsilon,\mu})$ is free over $k$. Moreover, there is a basis $\{w^{n(p+\epsilon)}:~p\in\bZ\}$ such that the action of $(\fg_{n,m},T^1)$ is given by
		\[Ew^{n(p+\epsilon)}=(\frac{1}{4mn}\mu+\frac{1}{2}(p+\epsilon))w^{n(p+1+\epsilon)};\]
		\[Fw^{n(p+\epsilon)}=(\frac{1}{2}\mu-mn(p+\epsilon))w^{n(p-1+\epsilon)};\]
		\[Hw^{n(p+\epsilon)}=n(p+\epsilon)w^{n(p+\epsilon)};\]
		\[t\cdot w^{n(p+\epsilon)}=t^{n(p+\epsilon)}w^{n(p+\epsilon)},\]
		where $t\in T^1$.
	\end{thm}
	\begin{proof}
		Proposition \ref{prop:compactpicture} and its proof imply that the restriction along $R(T^1)\to R(\fg_{n,m},T^1)$ gives rise to an isomorphism
		\[I^{\fg_{n,m},T^1}_{\fq,M}(k_{n\epsilon,\mu})=\Hom_{\fq,M}(R(\fg_{n,m},T^1),k_{n\epsilon,\mu})_{T^1}
		\cong\Hom_M(R(T^1),k_{n\epsilon})_{T^1}.\]
		Define $w^{n(p+\epsilon)}\in
		(\Hom_{\fq,M}(R(\fg_{n,m},T^1),k_{n\epsilon,\mu})_{T^1})_{n(p+\epsilon)}$ by
		\[w^{n(p+\epsilon)}(1\otimes p_\lambda)=\left\{\begin{array}{ll}
		1& (\lambda=n(p+\epsilon))\\
		0&({\rm otherwise}).
		\end{array} \right.\]
		Following Corollary \ref{cor:fundamentaltheoremofweakHeckealgebra}, we can compute the actions of $R(T^1)$, $E$ and $F$ as
		\[\begin{split}
		(p_\lambda w^{n(p+\epsilon)})(p_{\lambda'})
		&=w^{n(p+\epsilon)}(p_\lambda\cdot p_{\lambda'})\\
		&=\left\{\begin{array}{ll}
		1& (\lambda=\lambda'=n(p+\epsilon))\\
		0&({\rm otherwise}).
		\end{array} \right.
		\end{split}\]
		\[\begin{split}
		(Ew^{n(p+\epsilon)})(p_{n(p+1+\epsilon)})
		&=\sum_{\lambda\in\bZ}w^{n(p+\epsilon)}((1\otimes p_{n(p+1+\epsilon)})\cdot(E\otimes p_\lambda))\\
		&=w^{n(p+\epsilon)}(E\otimes p_{n(p+\epsilon)})\\
		&=w^{n(p+\epsilon)}((-\frac{1}{4mn}X+\frac{1}{4mn}Y+\frac{1}{2n}H)\otimes p_{n(p+\epsilon)})\\
		&=\frac{1}{4mn}\mu+\frac{1}{2}(p+\epsilon)
		\end{split}\]
		\[\begin{split}
		(Fw^{n(p+\epsilon)})(p_{n(p-1+\epsilon)})
		&=\sum_{\lambda\in\bZ}w^{n(p+\epsilon)}((1\otimes p_{n(p-1+\epsilon)})\cdot(F\otimes p_\lambda))\\
		&=w^{n(p+\epsilon)}(F\otimes p_{n(p+\epsilon)})\\
		&=w^{n(p+\epsilon)}((\frac{1}{2}X+\frac{1}{2}Y-mH)\otimes p_{n(p+\epsilon)})\\
		&=\frac{1}{2}\mu-mn(p+\epsilon)
		\end{split}\]
		This completes the proof.
	\end{proof}
	\begin{rem}
		The counit $I^{\fg_{n,m},T^1}_{\fq,M}(k_{n\epsilon,\mu})\to k_{n\epsilon,\mu}$ is given by $w^{n(p+\epsilon)}\mapsto 1$.
	\end{rem}
\begin{cor}\label{cor:casimir}
The element $m^2H^2 +mnEF+mnFE$ of the enveloping algebra of $\fg_{n,m}$ acts on $I^{\fg_{n,m},T^1}_{\fq,M}(k_{n\epsilon,\mu})$ by $\frac{1}{4}\mu(\mu-2mn)$.
\end{cor}
\begin{proof}
	This follows by a straightforward computation.
\end{proof}
For a digression, we discuss the conjugation operation.
In general, let $k\to k'$ be a flat homomorphism of commutative rings, $\sigma$ be a $k$-algebra automorphism of $k'$, and $(\fg,K)$ be a Harish-Chandra pair over $k$. Then we define the functor
\[(-)^\sigma:(\fg\otimes_k k',K\otimes_k k')\cmod
\to (\fg\otimes_k k',K\otimes_k k')\cmod\]
by the base change along $\sigma^{-1}$ and the canonical identification
\[((\fg\otimes_k k')\otimes_{k'} k',
(K\otimes_k k')\otimes_{k'} k')\cong
(\fg\otimes_k k',K\otimes_k k').\]

For a commutative ring $k$, define a $k$-algebra automorphism $\bar{\ }$ of $k\left[\sqrt{-1}\right]$ by $\overline{\sqrt{-1}}=-\sqrt{-1}$. We will call $\bar{\ }$ the conjugation. We shall discuss how the induced module of Theorem \ref{thm:fractionalmodel} behaves under the conjugation of $k\left[\sqrt{-1}\right]$. The flat base change theorem (\cite{MR3853058} Variant G (2)) shows
\[\overline{I^{\fg_{n,m},T^1}_{\fq,M}(
	k\left[\sqrt{-1}\right]_{n\epsilon,\mu})}
\cong I^{\fg_{n,m},T^1}_{\fq,M}(
k\left[\sqrt{-1}\right]_{n\epsilon,\bar{\mu}}).\]
\begin{prop}\label{prop:descent}
	\begin{enumerate}
		\renewcommand{\labelenumi}{(\arabic{enumi})}
		\item Suppose that $k$ is a field without $\sqrt{-1}$ such that the characteristic of $k$ does not divide $2mn$. Then the $(\fg_{n,m},T^1)$-module $I^{\fg_{n,m},T^1}_{\fq,M}(
		k\left[\sqrt{-1}\right]_{n\epsilon,\mu})$ over $k\left[\sqrt{-1}\right]$ is isomorphic to its conjugation if and only if $\mu-mn\in k\cup k\sqrt{-1}$. Moreover, if $\mu-mn\in k\cup k\sqrt{-1}$, $I^{\fg_{n,m},T^1}_{\fq,M}(
		k\left[\sqrt{-1}\right]_{n\epsilon,\mu})$ is defined over $k$.
		\item Suppose that $k=\bZ\left[1/2mn\right]$. Then the $(\fg_{n,m},T^1)$-module \[I^{\fg_{n,m},T^1}_{\fq,M}(
		k\left[\sqrt{-1}\right]_{n\epsilon,\mu})\]
		over $k\left[\sqrt{-1}\right]$ is isomorphic to its conjugation if and only if $\mu\in k$.
	\end{enumerate}
\end{prop}
For the proof of (2), we need elementary results on Gaussian integers:
\begin{lem}\label{lem:gaussianprime}
	\begin{enumerate}
		\renewcommand{\labelenumi}{(\arabic{enumi})}
		\item Let $\ell$ be a Gaussian prime such that $\ord_\ell (2mn)=0$. Then $\ord_\ell$ takes nonnegative values on $\bZ\left[1/2mn,\sqrt{-1}\right]$. In particular, we have $\ord_\ell$ is zero on $\bZ\left[1/2mn,\sqrt{-1}\right]^\times$.  
		\item Let $z$ and $w$ be nonzero elements of $\bZ\left[1/2mn,\sqrt{-1}\right]$. Then there exist $p\in\bZ$ and a Gaussian prime $\ell$ such that $\ord_\ell z=\ord_\ell (2mn)=0$ and $\ord_\ell (w-p)\neq 0$.
	\end{enumerate}
\end{lem}
\begin{proof}
	Part (1) is clear: let $\ell$ be a Gaussian prime with $\ord_\ell (2mn)=0$. Then every element of $\bZ\left[1/2mn,\sqrt{-1}\right]$ is expressed as $\frac{z}{(2mn)^a}$, where $z\in \bZ\left[\sqrt{-1}\right]$ and $a\in\bZ$. Since $\ord_\ell (2mn)=0$, we have
	\[\ord_\ell \frac{z}{(2mn)^a}=\ord_\ell z\geq 0.\]
	
	We next prove (2). Let $z$ and $w$ be nonzero elements of $\bZ\left[1/2mn,\sqrt{-1}\right]$. Let $N$ denote the norm of $\bQ(\sqrt{-1})/\bQ$. Let $\ell$ be a Gaussian prime whose norm $N(\ell)$ is a sufficiently large prime so that $\ord_\ell z=\ord_\ell (2mn)=0$. Choose a positive integer $a$ such that $(2mn)^a w\in\bZ\left[\sqrt{-1}\right]$. Since $\ell$ does not divide $2mn$, there exists a Gaussian integer $w'\in \bZ\left[\sqrt{-1}\right]$ such that $(2mn)^a w\equiv (2mn)^a w'\pmod \ell$.
	
	Since the canonical map $\bZ\to\bZ\left[\sqrt{-1}\right]/(\ell)$ is surjective, one can find an integer $p$ such that $w'\equiv p\pmod\ell$. The assertion now follows since
	 \[\begin{split}
	 	\ord_\ell(w-p)
	 	&=\ord_\ell (2mn)^a(w-p)\\
	 	&=\ord_\ell
	 	((2mn)^aw-(2mn)^a w'+(2mn)^aw'-(2mn)^ap)\\
	 	&\geq \min\{\ord_\ell ((2mn)^a w-(2mn)^a w'),
	 	\ord_\ell (2mn)^a (w'-p)\}\\
	 	&=\min\{\ord_\ell ((2mn)^a w-(2mn)^a w'),
	 	\ord_\ell (w'-p)\}\\
	 	&>0.
	 \end{split}\]
\end{proof}
\begin{proof}[Proof of Proposition \ref{prop:descent}]
	For (1), let $k$ be a field over $\bZ\left[1/2mn\right]$ without $\sqrt{-1}$. In particular, $k\left[\sqrt{-1}\right]$ is a field. It is clear that $I^{\fg_{n,m},T^1}_{\fq,M}(
	k\left[\sqrt{-1}\right]_{n\epsilon,\mu})$ is defined over $k$ if $\mu \in k$. Suppose that $\mu-mn\in k\sqrt{-1}$. We may assume $\mu-mn$ is nonzero. Then $\frac{1}{4mn}\mu+\frac{1}{2}(p+\epsilon)$ is nonzero for every integer $p$. We define a nonzero element $a_p$ of $k\left[\sqrt{-1}\right]$ for $p\in \bZ$ by $a_0=1$ and the recursion formula
	\[a_{p+1}=(\mu+2mn(p+\epsilon))a_p.\]
	Set $w^{n(p+\epsilon)}_r=a_p w^{n(p+\epsilon)}$. Then we have
	\[\begin{split}
		Ew^{n(p+\epsilon)}_r
		&=a_p Ew^{n(p+\epsilon)}\\
		&=a_p(\frac{1}{4mn}\mu+\frac{1}{2}(p+\epsilon))w^{n(p+1+\epsilon)}\\
		&=\frac{a_p}{a_{p+1}}(\frac{1}{4mn}\mu+\frac{1}{2}(p+\epsilon))w^{n(p+1+\epsilon)}_r\\
		&=\frac{1}{4mn}w^{n(p+1+\epsilon)}_r
	\end{split}\]
	\[\begin{split}
		Fw^{n(p+\epsilon)}_r
		&=a_p Fw^{n(p+\epsilon)}\\
		&=a_p(\frac{1}{2}\mu-mn(p+\epsilon))w^{n(p-1+\epsilon)}\\
		&=\frac{a_p}{a_{p-1}}(\frac{1}{2}\mu-mn(p+\epsilon))w^{n(p-1+\epsilon)}_r\\
		&=\frac{1}{2}(\mu+2mn(p-1+\epsilon))(\mu-2mn(p+\epsilon))w^{n(p+1+\epsilon)}_r
	\end{split}\]
	for $p\in\bZ$. Since $\mu-mn\in k\sqrt{-1}$,
	\begin{flalign*}
		&(\mu+2mn(p-1+\epsilon))(\mu-2mn(p+\epsilon))\\
		&=(\mu-mn+mn(2p-1+2\epsilon))(\mu-mn-mn(2p-1+2\epsilon))
	\end{flalign*}
	 belongs to $k$. This shows that $\oplus_{p\in\bZ} kw^{n(p+\epsilon)}_r$ is a $k$-form of $I^{\fg_{n,m},T^1}_{\fq,M}(
	k\left[\sqrt{-1}\right]_{n\epsilon,\mu})$.
	
	Conversely, suppose that $I^{\fg_{n,m},T^1}_{\fq,M}(
	k\left[\sqrt{-1}\right]_{n\epsilon,\mu})$ is isomorphic to its conjugation. Then Corollary \ref{cor:casimir} implies $\mu(\mu-2mn)=\bar{\mu}(\bar{\mu}-2mn)$. Arrange this equality to get $(\mu-\bar{\mu})(\mu+\bar{\mu}-2mn)=0$. Since $k\left[\sqrt{-1}\right]$ is a domain, at least one of the equalities $\mu=\bar{\mu}$ and $\mu+\bar{\mu}=2mn$ holds. The assertion $\mu-mn\in k\cup k\sqrt{-1}$ now follows since $2$ is a unit of $k$.
	
	We next prove (2). Put $k=\bZ\left[1/2mn\right]$. The ``if'' direction is obvious. Suppose that $I^{\fg_{n,m},T^1}_{\fq,M}(
	k\left[\sqrt{-1}\right]_{n\epsilon,\mu})$ is isomorphic to its conjugation. A similar argument to the ``only if'' direction of (1) shows at least one of the conditions $\mu-mn\in k\cup k\sqrt{-1}$ holds. 
	
	Suppose $\mu\in mn+k\sqrt{-1}\setminus{\{mn\}}$. Take an isomorphism
	\[f:I^{\fg_{n,m},T^1}_{\fq,M}(
	k\left[\sqrt{-1}\right]_{n\epsilon,\mu})
	\cong I^{\fg_{n,m},T^1}_{\fq,M}(
	k\left[\sqrt{-1}\right]_{n\epsilon,\bar{\mu}}).\]
	Then there exists a unit $a_p\in k\left[\sqrt{-1}\right]^\times$ for each $p\in\bZ$ such that
	\[f(w^{n(p+\epsilon)})=a_pw^{n(p+\epsilon)}.\]
	Compare the actions of $F$ to get
		\begin{equation}\label{eq:recursion}
			(\frac{1}{2}\mu-mn(p+\epsilon))
			a_{p-1}
			=(\frac{1}{2}\bar{\mu}-mn(p+\epsilon))a_p.
		\end{equation}
	In view of Lemma \ref{lem:gaussianprime} (2), we can find $p\in\bZ$ and a Gaussian prime $\ell$ such that
	\begin{subequations}\label{subeq:ord1}
		\begin{equation}
			\ord_\ell (2mn)=\ord_\ell(\mu-\bar{\mu})=0
		\end{equation}
		\begin{equation}
			\ord_\ell(\frac{1}{2mn}\mu-p-\epsilon)>0.
		\end{equation}
	\end{subequations}
	In particular, we obtain
	\begin{subequations}\label{subeq:ord2}
		\begin{equation}
			\ord_\ell(\frac{1}{2}\mu-mn(p+\epsilon))>0.
		\end{equation}
		\begin{equation}
			\ord_\ell(\frac{1}{2}\bar{\mu}-mn(p+\epsilon))
			=\ord_\ell(-\frac{1}{2}(\mu-\bar{\mu})+\frac{1}{2}\mu-mn(p+\epsilon))
			=0
		\end{equation}
	\end{subequations}
	from \eqref{subeq:ord1}. Finally, \eqref{subeq:ord2} contradicts to \eqref{eq:recursion} since $\ord_\ell a_p=\ord_\ell a_{p+1}=0$ (Lemma \ref{lem:gaussianprime} (1)).
\end{proof}
\begin{rem}
	Let $k$ be as in (1). Suppose that $\mu-mn$ is a nonzero element of $k\sqrt{-1}$. Then the $k$-form of $I^{\fg_{n,m},T^1}_{\fq,M}(
	k\left[\sqrt{-1}\right]_{n\epsilon,\mu})$ cannot be isomorphic to $I^{\fg_{n,m},T^1}_{\fq,M}(
	k_{n\epsilon',\mu'})$ for any $\epsilon\in\{0,\frac{1}{n},\frac{2}{n},\cdots,\frac{n-1}{n}\}$ and $\mu'\in k$. In fact, if such $\epsilon'$ and $\mu'$ exist, Corollary \ref{cor:casimir} implies $\mu(\mu-2mn)=\mu'(\mu'-2mn)$. Since $k\left[\sqrt{-1}\right]$ is a field, $\mu'$ is either $\mu$ or $2mn-\mu$. However, none of $\mu$ and $2mn-\mu$ belong to $k$.
\end{rem}
	Other new phenomena occur when $2mn$ is not invertible in $k$. Put $k=\bZ$. Take $\epsilon\in\{0,\frac{1}{n},\cdots,\frac{n-1}{n}\}$ and $\mu\in\bZ$. Define $\bZ_{n\epsilon,\mu}$ in a similar way.
	For each integer $\lambda\in\bZ$, we have a homomorphism
	\[\begin{split}
	\Hom_{T^1}(\bZ_\lambda,I^{\fg_{n,m},T^1}_{\fq,M}(\bZ_{n\epsilon,\mu}))
	&\cong\Hom_{\fg_{n,m},T^1}(\ind^{\fg_{n,m}}_{\ft^1} \bZ_\lambda,I^{\fg_{n,m},T^1}_{\fq,M}(\bZ_{n\epsilon,\mu}))\\
	&\cong\Hom_{\fq,M}({\rm ind}^{\fg_{n,m}}_{\ft^1} \bZ_\lambda,\bZ_{n\epsilon,\mu})\\
	&\to\Hom_M(\bZ_\lambda,\bZ_{n\epsilon})
	\end{split}\]
	which is compatible with the sequence of isomorphisms in the proof of Proposition \ref{prop:compactpicture}. Since the localization homomorphism $\bZ\to\bZ\left[1/2mn\right]$ is injective, $\Hom_{\fq,M}({\rm ind}^{\fg_{n,m}}_{\ft^1} \bZ_\lambda,\bZ_{n\epsilon,\mu})$ is bijective to the set of $M$-homomorphisms $\bZ_\lambda\to\bZ_{n\epsilon}$ which (uniquely) extend to a $(\fq,M)$-homomorphism ${\rm ind}^{\fg_{n,m}}_{\ft^1} \bZ_\lambda\to\bZ_{n\epsilon,\mu}$. For a more explicit description, we analyze the bijection
	\[\Hom_{\fq,M}({\rm ind}^{\fg_{n,m}}_{\ft^1} \bZ\left[1/2mn\right]_\lambda,\bZ\left[1/2mn\right]_{n\epsilon,\mu})
	\cong\Hom_M(\bZ\left[1/2mn\right]_\lambda,\bZ\left[1/2mn\right]_{n\epsilon}).\]
	We may assume that $\lambda$ is of the form $n(p+\epsilon)$ for some integer $p\in\bZ$; otherwise $\Hom_M(\bZ\left[1/2mn\right]_\lambda,\bZ\left[1/2mn\right]_{n\epsilon})=0$. Let \[\varphi\in\Hom_M(\bZ\left[1/2mn\right]_{n(p+\epsilon)},\bZ\left[1/2mn\right]_{n\epsilon}).\]
	In view of \eqref{eq:Iwasawa}, the extension is given by
	\[\begin{split}
	\varphi(E^{s+1}\otimes 1)
	&=\varphi((-\frac{1}{4mn}X+\frac{1}{4mn}Y+\frac{1}{2n}H)E^s\otimes 1)\\
	&=(\frac{1}{4mn}\mu+\frac{1}{2}(s+p+\epsilon))\varphi(E^s\otimes 1)
	\end{split}\]
	\[\begin{split}
	\varphi(F^{s+1}E^t\otimes 1)
	&=\varphi((\frac{1}{2}X+\frac{1}{2}Y-mH)(F^sE^t\otimes 1))\\
	&=(\frac{1}{2}\mu+mn(s-t-p-\epsilon))\varphi(F^sE^t\otimes 1)
	\end{split}\]
	for nonnegative integers $s,t\geq 0$.
	\begin{lem}\label{lem:elementarynumbertheory}
		We have
		\[\max \bigl\{\sum_{q=1}^{s}(1-\ord_2 q):~s\geq 0\bigr\}=\infty.\]
	\end{lem}
	\begin{proof}
		Put $s=2^a-1$ for some nonnegative integer $a$. Then
		\[\begin{split}
		\sum_{q=1}^{s}(1-\ord_2 q)
		&=2^a-1-\sum_{b=1}^\infty\left[\frac{2^a-1}{2^b}\right]\\
		&=2^a-1-\sum_{b=1}^{a}(2^{b-1}-1)\\
		&=a.
		\end{split}\]
		The assertion now follows.
	\end{proof}
	\begin{thm}\label{thm:partlyvanishing}
		The $(\fg_{n,m},T^1)$-module $I^{\fg_{n,m},T^1}_{\fq,M}(\bZ_{n\epsilon,\mu})$ is nonzero if and only if $\frac{1}{2mn}\mu+\epsilon\in\bZ$. Moreover, if $\frac{1}{2mn}\mu+\epsilon\in\bZ$, there is a canonical isomorphism
		\[I^{\fg_{n,m},T^1}_{\fq,M}(\bZ_{n\epsilon,\mu})=\oplus_{p\leq-\frac{1}{2mn}\mu-\epsilon}\bZ 2^{M_p}w^{n(p+\epsilon)}\subset I^{\fg_{n,m},T^1}_{\fq,M}(\bZ\left[1/2mn\right]_{n\epsilon,\mu}),\]
		where for each integer $p\leq-\frac{1}{2mn}\mu-\epsilon$,
		\[M_p=\max \bigl\{-\sum_{q=0}^{s}\ord_2(\frac{1}{4mn}\mu+\frac{1}{2}(q+p+\epsilon)):~0\leq s\leq 
		-(p+\frac{1}{2mn}\mu+\epsilon+1)\bigr\}\cup\{0\}.\]
	\end{thm}
	\begin{proof}
		Let $\varphi$ be a nonzero element of $\Hom_M(\bZ_{n(p+\epsilon)},\bZ_{n\epsilon,\mu})$. 
		Suppose that at least one of the following conditions fails:
		\begin{enumerate}
			\item[(i)] $\frac{1}{2mn}\mu+\epsilon\in\bZ$;
			\item[(ii)] $p\leq-\frac{1}{2mn}\mu-\epsilon$.
		\end{enumerate}
		Then $\varphi(E^s\otimes 1)$ never vanishes for $s\geq 0$. If (i) is satisfied, and (ii) fails then
		Lemma \ref{lem:elementarynumbertheory} implies
		\[\min \bigl\{\sum_{q=0}^{s-1}\ord_2(\frac{1}{4mn}\mu+\frac{1}{2}(q+p+\epsilon)):~s\geq 0\bigr\}=-\infty.\]
		Suppose that (i) fails. Then there exists a prime number $\ell$ such that $\ord_\ell(\frac{1}{2mn}\mu+\epsilon)<0$. This implies that for any $q\geq 0$,
		\[\begin{split}
		\ord_\ell(\frac{1}{2}(\frac{1}{2mn}\mu+(q+p+\epsilon)))
		&=-\ord_\ell 2+\ord_\ell(\frac{1}{2mn}\mu+(q+p+\epsilon))\\
		&=-\ord_\ell 2
		+\ord_\ell(\frac{1}{2mn}\mu+\epsilon)<0,
		\end{split}\]
		and
		\[\min \bigl\{\sum_{q=0}^{s-1}\ord_\ell
		(\frac{1}{4mn}\mu+\frac{1}{2}(q+p+\epsilon)):~s\geq 0\bigr\}=-\infty.\]
		Hence $\varphi$ never extends to a $(\fq,M)$-homomorphism $\ind^{\fg_{n,m}}_{\ft^1}\bZ_{n(p+\epsilon)}\to\bZ_{n\epsilon,\mu}$.
		
		If (i) and (ii) are satisfied, $s_0=-\frac{1}{2mn}\mu-\epsilon-p$ is a nonnegative integer, and $\varphi(E^{s_0+1}\otimes 1)=0$.
		Since $\mu$ is even from (i),
		\[\varphi(F^{s+1}E^t\otimes 1)\in\bZ\varphi(F^sE^t\otimes 1)\]
		for all $s,t\geq 0$. Hence $\varphi$ has an extension if and only if $\varphi(E^{s'}\otimes 1)\in\bZ$ for $1\leq s'\leq s_0$ in this case. It is characterized by
		\[\ord_2\varphi(1)+\sum_{q=0}^{s'-1}\ord_2(\frac{1}{4mn}\mu+\frac{1}{2}(q+p+\epsilon))\geq 0\]
		for all $s'$. This completes the proof.
	\end{proof}
	\begin{rem}\label{rem:choiceofq}
		There are other choices of $\bZ$-forms of $\fq_\bC$. In fact, a $\bZ$-form is determined by 
		the "Levi part" and a submodule of $\bZ(-2mnE+F+2mH)$ as a nilradical. For instance, the maximal $\bZ$-form is
		\[\bZ(-2mnE+F+2mH)\oplus\bZ(-2nE+H).\]
		For each choice, we can think of $\mu\in\bZ$ as the parameter $\mu\in\bC$ for $(\fq_\bC,T^1)\subset(\fsl_2,T^1)$ via $\alpha$. In other words, to fix a $\bZ$-form of $\fq_\bC$ is to fix a $\bZ$-form $\bZ_{n\epsilon,\mu}$ of the $(\fq_\bC,T^1)$-module $\bC_{n\epsilon,\mu}$ for $\mu\in\bC$ in our context. Formally, Theorem \ref{thm:partlyvanishing} is independent of the choice of $\bZ$-forms except the coefficient of $\mu$.
	\end{rem}

	We can obtain similar results by replacing the realization homomorphism $\alpha_{\frac{1}{2}}$. Let us think of $(\fg_{n,m},T^1,\alpha_{mn})$. Set
	\[\fq'=\bZ (-E+2mnF+2mH)\oplus \bZ (E+2mnF).\]
	Then $\fq'$ and $M$ form a subpair of $(\fg_{n,m},T^1)$. Define a $(\fq',M)$-module $k_{n\epsilon,\mu}$ in a similar way for a $\bZ\left[1/2mn\right]$-algebra $k$. Then one can obtain the following results by similar computations:
	\begin{var}\label{var:fractionalmodel1}
		Let $k$ be a $\bZ\left[1/2mn\right]$-algebra, $\mu\in k$, and $\epsilon\in\{0,\frac{1}{n},\frac{2}{n},\cdots,\frac{n-1}{n}\}$. Then the $(\fg_{n,m},T^1)$-module $I^{\fg_{n,m},T^1}_{\fq',M}(k_{n\epsilon,\mu})$ is free over $k$. Moreover, there is a basis $\{(w')^{n(p+\epsilon)}:~p\in\bZ\}$ such that the action of $(\fg_{n,m},T^1)$ is given by
		\[E(w')^{n(p+\epsilon)}=(\frac{1}{2}\mu+mn(p+\epsilon))(w')^{n(p+1+\epsilon)};\]
		\[F(w')^{n(p+\epsilon)}=(\frac{1}{4mn}\mu-\frac{1}{2}(p+\epsilon))(w')^{n(p-1+\epsilon)};\]
		\[H(w')^{n(p+\epsilon)}=n(p+\epsilon)(w')^{n(p+\epsilon)};\]
		\[t\cdot (w')^{n(p+\epsilon)}=t^{n(p+\epsilon)}(w')^{n(p+\epsilon)},\]
		where $t\in T^1$. In particular, $m^2H^2 +mnEF+mnFE$ acts on $I^{\fg_{n,m},T^1}_{\fq',M}(k_{n\epsilon,\mu})$ by $\frac{1}{4} \mu (\mu - 2 m n)$.
	\end{var}
\begin{var}\label{var:descent1}
	\begin{enumerate}
		\renewcommand{\labelenumi}{(\arabic{enumi})}
		\item Suppose that $k$ is a field without $\sqrt{-1}$ such that the characteristic of $k$ does not divide $2mn$. Then the $(\fg_{n,m},T^1)$-module $I^{\fg_{n,m},T^1}_{\fq',M}(
		k\left[\sqrt{-1}\right]_{n\epsilon,\mu})$ over $k\left[\sqrt{-1}\right]$ is isomorphic to its conjugation if and only if $\mu-mn\in k\cup k\sqrt{-1}$. Moreover, if $\mu-mn\in k\cup k\sqrt{-1}$, $I^{\fg_{n,m},T^1}_{\fq',M}(
		k\left[\sqrt{-1}\right]_{n\epsilon,\mu})$ is defined over $k$.
		\item Suppose that $k=\bZ\left[1/2mn\right]$. Then the $(\fg_{n,m},T^1)$-module \[I^{\fg_{n,m},T^1}_{\fq',M}(
		k\left[\sqrt{-1}\right]_{n\epsilon,\mu})\]
		over $k\left[\sqrt{-1}\right]$ is isomorphic to its conjugation if and only if $\mu\in k$.
	\end{enumerate}
\end{var}
	We can also define a $(\fq',M)$-module $\bZ_{n\epsilon,\mu}$ as well.
	\begin{var}\label{var:partlyvanishing}
		Let $\epsilon\in\{0,\frac{1}{n},\frac{2}{n},\cdots,\frac{n-1}{n}\}$ and $\mu\in\bZ$. Then $I^{\fg_{n,m},T^1}_{\fq',M}(\bZ_{n\epsilon,\mu})$ is nonzero if and only if $\frac{1}{2mn}\mu-\epsilon\in\bZ$.
		Moreover, if $\frac{1}{2mn}\mu-\epsilon\in\bZ$ then there is a canonical isomorphism
		\[I^{\fg_{n,m},T^1}_{\fq',M}(\bZ_{n\epsilon,\mu})=\oplus_{p\geq\frac{1}{2mn}\mu-\epsilon}\bZ 2^{N_p}(w')^{n(p+\epsilon)},\]
		where for each integer $p\geq\frac{1}{2mn}\mu-\epsilon$,
		\[N_p=\max \bigl\{-\sum_{q=0}^{s}\ord_2(\frac{1}{4mn}\mu+\frac{1}{2}(q-p-\epsilon)):~0\leq s\leq 
		p-\frac{1}{2mn}\mu+\epsilon-1\bigr\}\cup\{0\}.\]
	\end{var}
	Consider $(\fg_{n,2n},T^1,\alpha_n)$.
	Set
	\[\fq''=\bZ(-E+F+2H)\oplus\bZ(E+F).\]
	Then we define $\bZ_{n\epsilon,\mu}$ and $k_{n\epsilon,\mu}$ for a $\bZ\left[1/2n\right]$-algebra\footnote{More generally, we can work with a $\bZ\left[1/2\right]$-algebra $k$ such that $n$ is regular in $k$.} $k$ in a similar way again.
	\begin{var}\label{var:fractionalmodel2}
		Let $k$ be a $\bZ\left[1/2n\right]$-algebra, $\mu\in k$, and $\epsilon\in\{0,\frac{1}{n},\frac{2}{n},\cdots,\frac{n-1}{n}\}$. Then $(\fg_{n,2n},T^1)$-module $I^{\fg_{n,2n},T^1}_{\fq'',M}(k_{n\epsilon,\mu})$ is free over $k$. Moreover, there is a basis $\{(w'')^{n(p+\epsilon)}:~p\in\bZ\}$ such that the action of $(\fg_{n,2n},T^1)$ is given by
		\[E(w'')^{n(p+\epsilon)}=(\frac{1}{2}\mu+n(p+\epsilon))(w'')^{n(p+1+\epsilon)};\]
		\[F(w'')^{n(p+\epsilon)}=(\frac{1}{2}\mu-n(p+\epsilon))(w'')^{n(p-1+\epsilon)};\]
		\[H(w'')^{n(p+\epsilon)}=n(p+\epsilon)(w'')^{n(p+\epsilon)};\]
		\[t\cdot (w'')^{n(p+\epsilon)}=t^{n(p+\epsilon)}(w'')^{n(p+\epsilon)},\]
		where $t\in T^1$. In particular, $2H^2+EF+FE$ acts on $I^{\fg_{n,2n},T^1}_{\fq'',M}(k_{n\epsilon,\mu})$ by\[\frac{1}{2}\mu (\mu-2n).\]
	\end{var}
\begin{var}\label{var:descent2}
	\begin{enumerate}
		\renewcommand{\labelenumi}{(\arabic{enumi})}
		\item Suppose that $k$ is a field without $\sqrt{-1}$ such that the characteristic of $k$ does not divide $2n$. Then the $(\fg_{n,2n},T^1)$-module \[I^{\fg_{n,2n},T^1}_{\fq'',M}(
		k\left[\sqrt{-1}\right]_{n\epsilon,\mu})\]
		over $k\left[\sqrt{-1}\right]$ is isomorphic to its conjugation if and only if
		\[\mu-n\in k\cup k\sqrt{-1}.\]
		Moreover, if $\mu-n\in k\cup k\sqrt{-1}$, $I^{\fg_{n,2n},T^1}_{\fq'',M}(
		k\left[\sqrt{-1}\right]_{n\epsilon,\mu})$ is defined over $k$.
		\item Suppose that $k=\bZ\left[1/2n\right]$. Then the $(\fg_{n,2n},T^1)$-module \[I^{\fg_{n,2n},T^1}_{\fq'',M}(
		k\left[\sqrt{-1}\right]_{n\epsilon,\mu})\]
		over $k\left[\sqrt{-1}\right]$ is isomorphic to its conjugation if and only if $\mu\in k$.
	\end{enumerate}
\end{var}
\begin{rem}
	Variant \ref{var:descent1} and Variant \ref{var:descent2} formally follow from Proposition \ref{prop:descent} by the isomorphisms \[I^{\fg_{n,m},T^1}_{\fq,M}(
	k_{n\epsilon,\mu})\cong I^{\fg_{n,m},T^1}_{\fq',M}(
	k_{n\epsilon,\mu});~w^{n(p+\epsilon)}\mapsto (2mn)^p (w')^{n(p+\epsilon)}\]
	\[I^{\fg_{n,2n},T^1}_{\fq,M}(
	k_{n\epsilon,2n\mu})\cong I^{\fg_{n,2n},T^1}_{\fq'',M}(
	k_{n\epsilon,\mu});~w^{n(p+\epsilon)}\mapsto (2n)^p (w'')^{n(p+\epsilon)}.\]
\end{rem}
	\begin{var}
		Let $\epsilon\in\{0,\frac{1}{n},\frac{2}{n},\cdots,\frac{n-1}{n}\}$ and $\mu\in\bZ$. Then $I^{\fg_{n,2n},T^1}_{\fq'',M}(\bZ_{n\epsilon,\mu})$ vanishes if and only if $\mu$ is odd. Moreover, if $\mu$ is even, there is a canonical isomorphism
		\[I^{\fg_{n,2n},T^1}_{\fq'',M}(\bZ_{n\epsilon,\mu})\cong\oplus_{p\in\bZ}\bZ(w'')^{n(p+\epsilon)}.\]
	\end{var}
	Note that we have recurrence formulas for the cases of $\fq'$, $\fq''$:
	\[\varphi(F^{s+1}\otimes 1)=(\frac{1}{4mn}\mu+\frac{1}{2}(s-p-\epsilon))\varphi(F^s\otimes 1)\]
	\[\varphi(E^{s+1}F^t\otimes 1)=(\frac{1}{2}\mu+mn(s-t+p+\epsilon))\varphi(E^sF^t\otimes 1)\]
	\[\varphi(F^{s+1}\otimes 1)=(\frac{1}{2}\mu+n(s-p-\epsilon))\varphi(F^s\otimes 1)\]
	\[\varphi(E^{s+1}F^t\otimes 1)=(\frac{1}{2}\mu+n(s-t+p+\epsilon))\varphi(E^sF^t\otimes 1).\]
	
	Though we cannot take the base change from $\bZ$ to the finite field $\bF_2$ since $M$ is singular, we can find another setting in an appropriate sense.
	\begin{ex}
		Put $n=2$, $m=1$, and $c=1$ in Theorem \ref{thm:classification}. Set
		\[\fq=\bZ(-E+F+H)\oplus\bZ(E+F).\]
		Then $\fq$ and $M$ eventually form a subpair of $(\fg_{2,1},T^1)$ over $\bF_2$-algebras $k$. We regard $(\fg_{2,1},T^1)$ and $(\fq,M)$ as Harish-Chandra pairs over $k$. Recall that we have a right adjoint functor $I^{\fg,K}_{\fq,M,w}:(\fq,M)\cmod_w\to(\fg,K)\cmod_w$ (see the proof of \cite{MR4007195} Theorem 2.2.8). Since $M\to T^1$ induces an isomorphism of their Lie algebras, the diagram
		\[\xymatrix{
			(\fq,M)\cmod\ar[r]^{I^{\fg,K}_{\fq,M}}\ar[d]_{\cJ_{\fq,M}}&(\fg,K)\cmod\ar[d]^{\cJ_{\fg,K}}\\
			(\fq,M)\cmod_w\ar[r]^{I^{\fg,K}_{\fq,M,w}}&(\fg,K)\cmod_w
		}\]
		is 2-commutative, where the vertical arrows are the canonical embeddings. In particular, we obtain $I^{\fg,T^1}_{\fq,M}(k_{n\epsilon,\mu})\cong\Hom_{\fq,M}(U(\fg)\otimes_k R(T^1),k_{n\epsilon,\mu})_{T^1}$ from Corollary \ref{cor:fundamentaltheoremofweakHeckealgebra}.
	\end{ex}
	\section{Variants for contraction families}
	In this section, we discuss a contraction analog of computations in the preceding sections. Let us review the contraction families introduced by \cite{10.1093/imrn/rny146} and \cite{10.1093/imrn/rny147}. Let $(\fg,K)$ be a Harish-Chandra pair over $\bC$, equipped with a $K$-equivariant involution $\theta$ of $\fg$. Write $\fg^{\theta=1}$ (resp.\ $\fg^{\theta=-1}$) for the eigenspace of $\theta$ with eigenvalue $1$ (resp.\ $-1$). Assume that $\fk$ maps to $\fg^{\theta=1}$ via the structure homomorphism $\fk\to\fg$. Then $\bfg=\fg\otimes_\bC \bC\left[z\right]$ is a Lie algebra over $\bC\left[z\right]$ for the bracket defined summandwisely by
	\[\left[\eta z^m,\xi z^n\right]=\left\{\begin{array}{ll}
	\left[\eta,\xi\right]z^{m+n+1}& (\eta,\xi\in\fg^{\theta=-1})\\
	\left[\eta,\xi\right]z^{m+n}&({\rm otherwise}).
	\end{array} \right.\]
	The Harish-Chandra pair $(\bfg,K\otimes_\bC \bC\left[z\right])$ over $\bC\left[z\right]$ is called a contraction family.
	
	Fix a positive integer $n$, and consider the Harish-Chandra pair $(\fsl_2,T^1)$ over $\bC$ associated to the $n$-cover of PU(1,1). In view of naturality of the construction of contraction families, the $\theta$-stable parabolic subpairs $(\fb_\bC,T^1)$ and $(\bar{\fb}_\bC,T^1)$ of $(\fsl_2,T^1)$ extend to subpairs
	\[(\bfb_\bC,T^1),\ (\bar{\bfb}_\bC,T^1)\subset(\bfsl_2,T^1).\]
	The following result is deduced from similar computations to Theorem \ref{thm:thetastableinduction}:
	\begin{thm}
		Let $\lambda$ be an integer. Then we have the following descriptions:
		\[\ind^{\bfsl_2}_{\bar{\bfb}_\bC}(\bC\left[z\right]_\lambda)
		=\oplus_{p\geq0} \bC\left[z\right]y_{\lambda+np}\]
		\[\pro^{\bfsl_2}_{\fb_\bC}(\bC\left[z\right]_\lambda)
		=\oplus_{p\geq 0}\bC\left[z\right]y^{\lambda+np}\]
		\[y_{\lambda-n}=0\]
		\[\left(\begin{array}{cc}
		0&1\\
		0&0\\
		\end{array}
		\right)y_{\lambda+np}=y_{\lambda+n(p+1)};\]
		\[\left(\begin{array}{cc}
		0&0\\
		1&0\\
		\end{array}
		\right)y_{\lambda+np}=-\frac{1}{n}pz(np-n+2\lambda)y_{\lambda-n(p-1)};\]
		\[\left(\begin{array}{cc}
		1&0\\
		0&-1\\
		\end{array}
		\right)y_{\lambda+np}=\frac{2}{n}(\lambda+np)y_{\lambda+np}\]
		\[t\cdot y_{\lambda+np}=t^{\lambda+np}y_{\lambda+np};\]
		\[y^{\lambda-n}=0;\]
		\[\left(\begin{array}{cc}
		0&1\\
		0&0\\
		\end{array}
		\right)y^{\lambda+np}=-\frac{z}{n}(p+1)(np+2\lambda)y^{m+n(p+1)};\]
		\[\left(\begin{array}{cc}
		0&0\\
		1&0\\
		\end{array}
		\right)y^{\lambda+np}=y^{\lambda+n(p-1)};\]
		\[\left(\begin{array}{cc}
		1&0\\
		0&-1\\
		\end{array}
		\right)y^{\lambda+np}=\frac{2}{n}(\lambda+np)y^{\lambda+np};\]
		\[t\cdot y^{\lambda+np}=t^{\lambda+np}y^{\lambda+np},\]
		where $t\in T^1$.
	\end{thm}
	On the other hand, the counterpart of real parabolic induction and principal series representations is nontrivial since $\fq_\bC$ is not $\theta$-stable. We suggest the
	$\bC\left[z\right]$-submodule $\bfq$ of $\bfsl_2$ spanned by the two elements
	\[X=\left(\begin{array}{cc}
	z&-z\\
	1&-z\\
	\end{array}
	\right)\]
	\[Y=\left(\begin{array}{cc}
	0&z\\
	1&0\\
	\end{array}
	\right)\]
	as a contraction model of $\fq$. It gives rise to a subpair $(\bfq,M_\bC\otimes_\bC\bC\left[z\right])$. Note that a similar issue to Theorem \ref{thm:partlyvanishing} appears in this setting. Namely, the Iwasawa decomposition only holds over $\bC\left[z^{\pm 1}\right]$:
	\[\bfq\oplus\ft^1\cong\bfsl_2;\]
	\[\left(\frac{-b+cz}{2z}X+\frac{b+cz}{2z}Y,
	\frac{2a+b-zc}{2}\left(\begin{array}{cc}
	1&0\\
	0&-1\\
	\end{array}
	\right)\right)\mapsto
	\left(\begin{array}{cc}
	a&b\\
	c&-a\\
	\end{array}
	\right).\]
	Notice also that $\bC\left[z\right]X$ is an ideal of $\bfq$ so that the projection $\bfq\to\bC\left[z\right]Y$ is a Lie algebra homomorphism. For $\epsilon\in\{0,\frac{1}{n},\frac{2}{n},\cdots,\frac{n-1}{n}\}$ and $\mu\in\bC\left[z^{\pm 1}\right]$, $\bC\left[z^{\pm 1}\right]=\bC\left[z^{\pm 1}\right]_{n\epsilon,\mu}$ is a $(\bfq,M_\bC\otimes_\bC\bC\left[z\right])$-module for
	\[X\cdot 1=0;\]
	\[Y\cdot 1=\mu;\]
	\[t\cdot 1= t^{n\epsilon},\]
	where $t\in M_\bC\otimes_\bC\bC\left[z\right]$. Notice that if $\mu\in\bC\left[z\right]$, it restricts to $\bC\left[z\right]$, which the resulting module will be denoted by $\bC\left[z\right]_{n\epsilon,\mu}$. We define a $(\bfq,M_\bC\otimes_\bC\bC\left[\sqrt{z}^{\pm 1}\right])$-module $\bC\left[\sqrt{z}^{\pm 1}\right]_{n\epsilon,\mu}$ for $\epsilon\in\{0,\frac{1}{n},\frac{2}{n},\cdots,\frac{n-1}{n}\}$ and $\mu\in\bC\left[\sqrt{z}^{\pm 1}\right]$ in a similar way. Let $\bar{\ }$ denote the $\bC\left[z^{\pm 1}\right]$-algebra automorphism of $\bC\left[\sqrt{z}^{\pm 1}\right]$ defined by $\sqrt{z}\mapsto -\sqrt{z}$ which will be called the conjugation. Taking the above remarks into account, one can prove the following result:
	\begin{thm}\label{thm:contractionvanishing}
		\begin{enumerate}
			\renewcommand{\labelenumi}{(\arabic{enumi})}
			\item There is a free
			$\bC\left[z^{\pm 1}\right]$-basis
			\[\{w^{n(p+\epsilon)}:~p\in\bZ\}\]
			of $I^{\bfsl_2,T^1}_{\bfq,M_\bC\otimes_\bC\bC\left[z\right]}(\bC\left[z^{\pm 1}\right]_{n\epsilon,\mu})$ on which $\bfsl_2$ and $T^1$ act as
			\[\left(\begin{array}{cc}
			0&1\\
			0&0\\
			\end{array}
			\right)w^{n(p+\epsilon)}=(\frac{1}{2z}\mu+p+\epsilon)w^{n(p+1+\epsilon)};\]
			\[\left(\begin{array}{cc}
			0&0\\
			1&0\\
			\end{array}
			\right)w^{n(p+\epsilon)}=(\frac{1}{2}\mu-z(p+\epsilon))w^{n(p-1+\epsilon)};\]
			\[\left(\begin{array}{cc}
			1&0\\
			0&-1\\
			\end{array}
			\right)w^{n(p+\epsilon)}=2(p+\epsilon)w^{n(p+\epsilon)};\]
			\[t\cdot w^{n(p+\epsilon)}=t^{n(p+\epsilon)}w^{n(p+\epsilon)},\]
			where $t\in T^1$.
			\item For $\mu\in\bC\left[z\right]\setminus z\bC\left[z\right]$, $I^{\bfsl_2,T^1}_{\bfq,M_\bC\otimes_\bC\bC\left[z^{\pm 1}\right]}(\bC\left[z^{\pm 1}\right]_{n\epsilon,\mu})$ vanishes.
			\item Assume $\mu\in z\bC\left[z\right]$. Then $\oplus_{p\in\bZ}\bC\left[z\right]w^{n(p+\epsilon)}$ is a $(\bfsl_2,T^1)$-submodule of $I^{\bfsl_2,T^1}_{\bfq,M_\bC\otimes_\bC\bC\left[z\right]}(\bC\left[z^{\pm 1}\right]_{n\epsilon,\mu})$ over $\bC\left[z\right]$. Moreover, we have an isomorphism
			\[\oplus_{p\in\bZ}\bC\left[z\right]w^{n(p+\epsilon)}\cong I^{\bfsl_2,T^1}_{\bfq,M_\bC\otimes_\bC\bC\left[z\right]}(\bC\left[z\right]_{n\epsilon,\mu}).\]
			\item For $\epsilon\in\{0,\frac{1}{n},\frac{2}{n},\cdots,\frac{n-1}{n}\}$ and $\mu\in\bC\left[\sqrt{z}^{\pm 1}\right]$, \[I^{\bfsl_2,T^1}_{\bfq,M_\bC\otimes_\bC\bC\left[\sqrt{z}^{\pm 1}\right]}(\bC\left[\sqrt{z}^{\pm 1}\right]_{n\epsilon,\mu})\]
			is isomorphic to its conjugation if and only if $\mu\in\bC\left[z^{\pm 1}\right]$.
		\end{enumerate}
	\end{thm}
	\begin{rem}
		Theorem \ref{thm:contractionvanishing} (3) implies the base change formula 
		\[I^{\bfsl_2,T^1}_{\bfq,M_\bC\otimes_\bC\bC\left[z\right]}(\bC\left[z\right]_{n\epsilon,\mu})\otimes_{\bC\left[z\right]}\bC\left[z^{\pm 1}\right]
		\cong I^{\bfsl_2,T^1}_{\bfq,M_\bC\otimes_\bC\bC\left[z^{\pm 1}\right]}(\bC\left[z^{\pm 1}\right]_{n\epsilon,\mu})\]
		for $\mu\in z\bC\left[z\right]$ without the conditions of \cite{MR3853058} Theorem 3.1.7.
	\end{rem}
	\begin{rem}
		The $(\bfsl_2,T^1)$-module $I^{\bfsl_2,T^1}_{\bfq,M_\bC\otimes_\bC\bC\left[z^{\pm 1}\right]}(\bC\left[z^{\pm 1}\right]_{n\epsilon,\mu})$ is generically irreducible in the sense that its base change to the algebraic closure $\overline{\bC(z)}$ is irreducible (see \cite{10.1093/imrn/rny146}, \cite{10.1093/imrn/rny147}) if and only if $\mu$ does not belong to $2z\bZ\pm 2z\epsilon$.
	\end{rem}
	\begin{rem}\label{rem:howtocontract}
		The construction of $\bfq$ is generalized in the following way: suppose that we are given a complex reductive Lie algebra $\fg$ and a parabolic subalgebra $\fq$ with abelian nilradical $\fu$. Fix a Cartan subalgebra contained in $\fq$. Then we obtain a Levi decomposition $\fq=\fl\oplus\fu$. Write $\bar{\fu}$ for the nilradical of the opposite parabolic subalgebra to $\fq$. Then $\fg$ and $\fl$ form a symmetric pair for the involution $\theta$ defined by
		\[\theta(x)=\left\{\begin{array}{ll}
		x& (x\in\fl)\\
		-x&(x\in\fu\oplus\bar{\fu}).
		\end{array} \right.\]
		Let $\bfg$ be the associated contraction family. Then we have an isomorphism
		\[\Phi:\fg\otimes_\bC\bC\left[z^{\pm 1}\right]\cong\bfg\otimes_{\bC\left[z\right]}\bC\left[z^{\pm 1}\right];\]
		\[\Phi(x)=\left\{\begin{array}{ll}
		x& (x\in\fq)\\
		z^{-1}x&(x\in\bar{\fu})
		\end{array} \right.\]
		of Lie algebras over $\bC\left[z^{\pm 1}\right]$.
		
		For a Lie subalgebra $\fa\subset\fg$, the image $\Phi(z\bC\left[z\right]\fa)$ is a Lie subalgebra of $\bfg$. Passing to this isomorphism, we can identify Theorem \ref{thm:contractionvanishing} (1) with Theorem \ref{thm:fractionalmodel}. More specifically, suppose that there is a semidirect decomposition $\fa=\fa_s\oplus\fa_n$ with $\fa_n$ being an ideal, and that the subalgebra $\fa_s$ is contained in $\fq$. Then we can find a larger subalgebra $\Phi(\bC\left[z\right]\fa_s\oplus z\bC\left[z\right]\fa_n)$.
	\end{rem}
	\begin{rem}
		In the previous section, we considered Harish-Chandra pairs $(\fg_{n,m},T^1)$. From the perspective of contraction families, we can regard each of them as the fiber of the contraction family over $\bZ\left[z\right]$ associated to $(\fg_{n,1},T^1)$ at $z=m$. Following this idea, we can think that $\fq,\fq'\subset\fg_{n,m}$ are obtained from the construction of Remark \ref{rem:howtocontract}. The maximal $\bZ$-form in Remark \ref{rem:choiceofq} is obtained by the latter construction in Remark \ref{rem:howtocontract}.
	\end{rem}

\end{document}